\documentclass{amsart}
\usepackage{amsmath}
\usepackage{amsthm}
\usepackage{hyperref}
\usepackage{amsfonts,graphics,amsthm,amsfonts,amscd,latexsym, tikz-cd}
\usepackage{epsfig}
\usepackage{flafter}
\usepackage{mathtools}
\usepackage{comment}
\usepackage{stmaryrd}

\usepackage{mathabx,epsfig}

\hypersetup{
    colorlinks=true,    
    linkcolor=blue,          
    citecolor=blue,      
    filecolor=blue,      
    urlcolor=blue           
}
\usepackage{tikz}
\usetikzlibrary{graphs,positioning,arrows,shapes.misc,decorations.pathmorphing}

\tikzset{
    >=stealth,
    every picture/.style={thick},
    graphs/every graph/.style={empty nodes},
}

\tikzstyle{vertex}=[
    draw,
    circle,
    fill=black,
    inner sep=1pt,
    minimum width=5pt,
]
\usepackage[position=top]{subfig}
\usepackage{amssymb}
\usepackage{color}

\calclayout

\usetikzlibrary{decorations.pathmorphing}
\tikzstyle{printersafe}=[decoration={snake,amplitude=0pt}]

\newcommand{\Proj}{\operatorname{Proj}}

\newcommand{\supp}{\operatorname{supp}}

\newcommand{\Spec}{\operatorname{Spec}}

\newcommand{\ord}{\operatorname{ord}}

\newcommand{\vol}{\operatorname{vol}}

\newcommand{\gr}{\operatorname{gr}}
\renewcommand{\AA}{\mathbb{A}}

\newcommand{\PP}{\mathbb{P}}

\newcommand{\QQ}{\mathbb{Q}}
\newcommand{\ZZ}{\mathbb{Z}}
\newcommand{\NN}{\mathbb{N}}
\newcommand{\RR}{\mathbb{R}}
\newcommand{\CC}{\mathbb{C}}

\newcommand{\GG}{\mathbb{G}}

\newcommand{\OO}{\mathcal{O}}

\newcommand{\mm}{\mathfrak{m}}

\def\O#1.{\mathcal {O}_{#1}}			
\def\pr #1.{\mathbb P^{#1}}				
\def\af #1.{\mathbb A^{#1}}			
\def\ses#1.#2.#3.{0\to #1\to #2\to #3 \to 0}	
\def\xrar#1.{\xrightarrow{#1}}			
\def\K#1.{K_{#1}}						
\def\bA#1.{\mathbf{A}_{#1}}			
\def\bM#1.{{\mathbf{M}}_{#1}}
\def\bN#1.{{\mathbf{N}}_{#1}}
\def\bL#1.{\mathbf{L}_{#1}}				
\def\bB#1.{\mathbf{B}_{#1}}				
\def\bK#1.{\mathbf{K}_{#1}}			
\def\subs#1.{_{#1}}					
\def\sups#1.{^{#1}}

\usepackage{tikz}
\usetikzlibrary{matrix,arrows,decorations.pathmorphing}

  \newtheorem{theorem}{Theorem}[section]
  \newtheorem{lemma}[theorem]{Lemma}
  \newtheorem{proposition}[theorem]{Proposition}
  \newtheorem{corollary}[theorem]{Corollary}
  
  \newtheorem{conjecture}[theorem]{Conjecture}

  \newtheorem{definition}[theorem]{Definition}
  \newtheorem{example}[theorem]{Example}

  \newtheorem{question}[theorem]{Question}
  \newtheorem{construction}[theorem]{Construction}

\newtheorem{remark}[theorem]{Remark}

\theoremstyle{remark}

\numberwithin{equation}{section}

\usepackage[all]{xy}

\begin{document}

\title[$\GG_m$-Equivariant Degenerations of del Pezzo Surfaces]{$\GG_m$-Equivariant Degenerations of del Pezzo Surfaces}

\author[J.~Peng]{Junyao Peng}
\address{Department of Mathematics, Princeton University, Fine Hall, Washington Road, Princeton, NJ 08544-1000, USA
}
\email{junyaop@princeton.edu}

\begin{abstract}
We study special $\GG_m$-equivariant degenerations of a smooth del Pezzo surface $X$ induced by valuations that are log canonical places of $(X,C)$ for a nodal anti-canonical curve $C$. We show that the space of special valuations in the dual complex of $(X,C)$ is connected and admits a locally finite partition into sub-intervals, each associated to a $\GG_m$-equivariant degeneration of $X$. This result is an example of higher rank degenerations of log Fano varieties studied by Liu-Xu-Zhuang, and verifies a global analog of a conjecture on Koll\'ar valuations raised by Liu-Xu. For del Pezzo surfaces with quotient singularities, we obtain a weaker statement about the space of special valuations associated to a normal crossing complement. 
\end{abstract}

\maketitle

\setcounter{tocdepth}{1} 
\tableofcontents

\section{Introduction}
For a normal projective variety $X$, a degeneration of $X$ is a variety $X'$ with a flat family $\mathcal{X}\to C$ over a smooth pointed curve $0\in C$, such that the general fiber $\mathcal{X}_t$ is isomorphic to $X$ for $t\neq 0$ and the central fiber $\mathcal{X}_0$ is isomorphic to $X'$. A degeneration is $\QQ$-Gorenstein if $K_{\mathcal{X}}$ is $\QQ$-Cartier. It is then natural to ask whether one can classify all $\QQ$-Gorenstein degenerations of $X$. This question is closely related to the local geometry of moduli spaces of varieties. When $X$ is a variety of general type, any $\QQ$-Gorenstein degeneration of $X$ is trivial, which implies that the KSBA compactification of moduli space of varieties of general type is always Deligne-Mumford \cite{KSB88}\cite{Ale94}\cite{Kol23}. 

When $X$ is the projective plane $\PP^2$, the classification of degenerations of $X$ is a non-trivial problem. Manetti \cite{Man91} classified all normal degenerations of $\PP^2$ with quotient singularities in terms of numerical invariants and quotient types. Hacking and Prokhorov \cite{HP10} proved that $\QQ$-Gorenstein degenerations of $\PP^2$ must be partial smoothings of weighted projective spaces $\PP(a^2,b^2,c^2)$, where $a,b,c$ satisfy the Markov equation
\[
a^2+b^2+c^2 = 3abc.
\]
Solutions of the Markov equations are generated via mutations of the form $(a,b,c) \mapsto(a,b,3ab-c)$ and thus form a tree structure. Based on results about flips on a family of surfaces \cite{KM92}\cite{Mor02}\cite{HTU17}, Urzúa and Zúñiga \cite{UZ23} studied the Markov tree and showed that different degenerations of $\PP^2$ are related to each other by a sequence of explicit birational operations. Furthermore, $\QQ$-Gorenstein degenerations of $\PP^2$ which admit a non-trivial torus action are known to be either $\PP(a^2,b^2,c^2)$ for a Markov triple $(a,b,c)$ or a hypersurface of a weighted projective space \cite{HKW24}. 

For general del Pezzo surfaces, $\QQ$-Gorenstein degenerations are studied via $\QQ$-Gorenstein smoothings of singular surfaces. Hacking and Prokhorov \cite{HP10} classified $\QQ$-Gorenstein smoothable del Pezzo surfaces with quotient singularities and Picard rank 1. Prokhorov \cite{Pro19} then extended this result to the case of log canonical singularities. Other results in this direction include the higher Picard rank case \cite{Pro15} and the case of Wahl singularities \cite{Wah81}\cite{Urz16}\cite{Hac16}.

The goal of this paper is to study $\GG_m$-equivariant degenerations of del Pezzo surfaces, which are known as test configurations in the algebraic theory of K-stability. Roughly speaking, a test configuration is a flat $\GG_m$-equivariant family $\mathcal{X}\to \AA^1$ with general fiber $\mathcal{X}_t\cong X$. Via the Rees construction, these $\GG_m$-equivariant degenerations are induced by certain divisorial valuations which can be characterized by singularities of the MMP \cite[Theorem 4.12]{Xu21}. In the proof of the optimal destabilization conjecture \cite[Theorem 1.2]{LXZ22}, the authors studied higher rank versions of $\GG_m$-equivariant degenerations of Fano varieties. These higher rank degenerations of a Fano variety $X$ are induced by certain \textit{special quasi-monomial} valuations over $X$, which can be thought of as limits of divisorial valuations in the dual complex of $(X,\Delta)$ for some $\QQ$-complement $\Delta\sim_{\QQ}-K_X$. These special quasi-monomial valuations have characterizations similar to the divisorial case \cite[Theorem 4.2]{LXZ22}. However, the structure of the space of special valuations is not so well-understood.

\begin{question}[\text{cf. \cite[Question 1.5]{LX24}}]\label{question:special_valuations}
{\em
Let $X$ be a klt Fano variety and $\Delta\sim_{\QQ}-K_X$ be a $\QQ$-complement of $X$. Let $r\in \NN$ be such that $rK_X$ is Cartier. Let $\mathcal{D}(X,\Delta)$ be the dual complex of $(X,\Delta)$. Let $\mathcal{D}^{KV}(X,\Delta)\subseteq \mathcal{D}(X,\Delta)$ be the subset consisting of quasi-monomial valuations $v$ (up to scaling) which satisfy the following conditions:
\begin{itemize}
\item The ring
\[
\bigoplus_{m\in r\NN} \gr_v^\bullet H^0(X, -mK_X) 
\]
is finitely generated as a graded algebra.
\item The scheme define by
\[
\Proj \left(\bigoplus_{m\in r\NN} \gr_v^\bullet H^0(X, -mK_X) \right)
\]
is a klt Fano variety.
\end{itemize}
We ask the following questions:
\begin{itemize}
\item[(1)] What is the set $\mathcal{D}^{KV}(X,\Delta)$?
\item[(2)] As a subspace of $\mathcal{D}(X,\Delta)$, what is the topological structure of $\mathcal{D}^{KV}(X,\Delta)$?
\item[(3)] What are the degenerations of $X$ induced by $v\in \mathcal{D}^{KV}(X,\Delta)$? In other words, what are the Fano varieties defined by
\[\Proj \left(\bigoplus_{m\in r\NN} \gr_v^\bullet H^0(X, -mK_X) \right)?
\]
\end{itemize}
}
\end{question}

In \cite{LX24}, the authors asked similar questions in the local setting of a klt singularity $(X,\Delta;x)$. They also studied the analogous spaces $\mathcal{D}^{KV}(X,\Delta;x)$ when the dual complex of $(X,\Delta)$ is one-dimensional, and conjectured that $\mathcal{D}^{KV}(X,\Delta;x)$ admits a natural triangulation with certain finiteness property \cite[Conjecture 1.8]{LX24}.
Furthermore, explicit examples of $\mathcal{D}^{KV}(X,\Delta)$  were studied in the case that $X = \PP^2$ and $\Delta$ is a nodal cubic \cite[Theorem 6.1]{LXZ22}, or $X = \PP^2$ and $\Delta$ is the sum of a smooth conic and a transversal line \cite[Proposition A.4]{ABBDILW23}. Motivated by these results, we propose the following conjecture:
\begin{conjecture}[\text{cf. \cite[Conjecture 1.8]{LX24}}]\label{conj:structure_of_special_valuations}
Let $(X,D)$ be a klt Fano pair and $\Delta\sim_{\QQ} -(K_X+D)$ be a $\QQ$-complement of $(X,D)$. Then there is a rational triangulation of $\mathcal{D}^{KV}(X, D+\Delta)$ (notation as in Question) which satisfies the following properties:
\begin{itemize}
\item In the interior of $\mathcal{D}^{KV}(X,D+\Delta)$, this triangulation is locally finite.
\item Let $C^{\circ}$ be an open simplex of this triangulation. Then any valuation (up to scaling) in $C^{\circ}$ induces a $\GG_m$-equivariant degenerations of $X$ into a klt Fano variety. Furthermore, any two valuations in $C^\circ$ induce isomorphic $\GG_m$-equivariant degenerations.
\end{itemize}
\end{conjecture}

The main purpose of this article is to answer Question~\ref{question:special_valuations} and Conjecture~\ref{conj:structure_of_special_valuations} when $X$ is a smooth del Pezzo surface and $\Delta\sim -K_X$ is a normal crossing divisor on $X$.

\begin{theorem}\label{mainthm:classification_of_special_valuations_on_del_pezzo_surfaces}
Let $X$ be a smooth del Pezzo surface of degree $d$. Let $C\sim -K_X$ be a normal crossing divisor on $X$. Suppose $C$ is nodal and $x\in C$ is a node. Let $v\in \text{QM}(X,C)$ be a quasi-monomial valuation over $X$. 
\begin{itemize}
\item [(a)] Suppose $C$ is irreducible. For $t>0$, let $v_t$ denote the quasi-monomial valuation centered at $x$ with weight $(1,t)$ along the two branches of $C$. Then $v$ is special over $X$ if and only if $d\geq 5$ and $v = v_t$ for some $t\in I_d:= \left(\frac{d-2-\sqrt{d^2-4d}}{2}, \frac{d-2+\sqrt{d^2-4d}}{2}\right)$.
\item [(b)] Suppose $C$ has two irreducible components $C_1, C_2$ and both are nef. Then $v$ is special over $X$ if and only if it is not one of the following cases:
\begin{itemize}
\item $d = 4$.
\item $d\geq 5$, $C_1^2 = 0$, and $v = \ord_{C_2}$.
\item $d\geq 5$, $C_2^2 = 0$, and $v = \ord_{C_1}$.
\end{itemize}
\item [(c)] Suppose $C$ has $k\geq 3$ irreducible components $C_1,\ldots, C_k$ and $C_i$ is nef for all $1\leq i\leq k$. Then $k = 3,4$ and $v$ is special. 
\item [(d)] Suppose that some irreducible component of $C$ is not nef. Let $\phi: X\to X'$ be a morphism which only contracts components of $C$ and every component of $C':= \phi(C)$ is nef. Then $v$ is special over $X$ if and only if $v$ is special over $X'$. 
\end{itemize}
\end{theorem}

\begin{theorem}\label{mainthm:structure_of_the_space_of_special_valuations}
Let $X$ be a smooth del Pezzo surface. Let $C\sim -K_X$ be a normal crossing divisor on $X$ which is not smooth. Let $\mathcal{D}(X,C)$ be the dual complex of $(X,C)$ and $\mathcal{D}^{KV}(X,C)$ be the subset of $\mathcal{D}(X,C)$ consisting of special valuations over $X$\footnote{By Remark~\ref{remark:dual_complex_qm_vals}, we may think of $\mathcal{D}(X,C)$ as the space of quasi-monomial valuations in $\text{QM}(X,C)$ up to scaling. Notice that any rescaling of a special valuation is still special since the associated graded anticanonical ring remains unchanged.}. Then
\begin{itemize}
\item[(a)] $\mathcal{D}(X,C)$ is homoemorphic to $S^1$. Under this identification, $\mathcal{D}^{KV}(X,C)$ is either empty, $S^1$, or an open interval $I\subsetneq S^1$. 
\item[(b)] If $\mathcal{D}^{KV}(X,C)$ is $S^1$, then there exists finitely many disjoint open intervals $I_1,\ldots, I_n\subseteq \mathcal{D}^{KV}(X,C)$ such that
\[
\mathcal{D}^{KV}(X,C) = \overline{I_1\cup I_2\cup\cdots\cup I_n}.
\]
Furthermore any two valuations in the same interval $I_k$ induce isomorphic $\GG_m$-equivariant degenerations of $X$. 

\item [(c)] If $\mathcal{D}^{KV}(X,C)$ is an open interval $I\subsetneq S^1$, then there exists countably many disjoint open intervals $\{I_k\subseteq I: k\in \ZZ\}$ such that
\[
\mathcal{D}^{KV}(X,C) = \overline{\bigcup_{k\in \ZZ} I_k}
\]
and every accumulation point of the set $\bigcup_{k\in \ZZ} \partial I_k$ lies in $\partial I$. Furthermore, any two valuations in the same interval $I_k$ induce isomorphic $\GG_m$-equivariant degenerations of $X$. 

\item [(d)] In (b) and (c), suppose $t\in \partial I_k$ for some $k$, then $t$ corresponds to a divisorial valuation $\ord_E$ on $X$. Furthermore, suppose $E$ is the $(p,q)$ weighted blow-up along two branches of $C$ at a node $x\in C$, then there exists a $(p,q)$-unicuspidal rational curve on $X$ which is well-formed with respect to $(C,x)$.
\end{itemize}
\end{theorem}

\begin{remark}
{\em
Unicuspidal curves as in Theorem~\ref{mainthm:structure_of_the_space_of_special_valuations}(d) also play an important role in the study of the ellipsoid embedding problem in symplectic geometry \cite{MS12}. They provide obstructions to the embedding of 4-dimensional symplectic ellipsoids into del Pezzo surfaces (considered as 4-dimensional symplectic varieties) \cite{MS24}. The classification of unicuspidal curves in \cite{CGHMP24} then implies that the ellipsoid embedding function has an infinite staircase pattern. It would be interesting if one can show a deeper connection between the ellipsoid embedding problem and the structure of special quasi-monomial valuations in more general cases.
}
\end{remark}

When $X$ is a del Pezzo surface with klt singularities, we prove a result about the local structure of the dual complex near a special divisorial valuation. Roughly speaking, it says that any valuation which is sufficiently close to a special divisorial valuation is also special.
\begin{theorem}\label{mainthm:local_structure_of_dual_complex_of_special_degenerations}
Let $X$ be a klt del Pezzo surface. Let $C\sim_{\QQ} -K_X$ be a divisor on $X$ such that $(X,C)$ is log canonical. Suppose $(X,C)$ is a normal crossing pair when restricted to some open subset $U\subseteq X$. Suppose $E$ is a divisor over $X$ such that 
\begin{itemize}
\item $E$ is a log canonical place of $(X,C)$.
\item The center of $E$ on $X$ is a closed point $x\in U$.
\item $E$ is special over $X$.
\end{itemize}
Let $\mathcal{X}\to \AA^1$ be the special test configuration of $X$ induced by $E$ and $\mathcal{C}\subseteq \mathcal{X}$ be the closure of $C\times (\AA^1-\{0\})$. Let $(X',C')$ be the fiber of $(\mathcal{X},\mathcal{C})$ over $0\in \AA^1$. Let $\ord_{E'}$ be the canonical (up to scaling) divisorial valuation on $X'$ induced by the $\GG_m$ action on $X'$. Then the following statements hold.
\begin{itemize}
\item[(a)] $(X',C')$ is log canonical and $E'$ is a log canonical place of $(X',C')$. 
\item[(b)] There exists a homeomorphism between an open neighborhood $I\subseteq \mathcal{D}(X,C)$ containing $\ord_E$ and an open neighborhood $I'\subseteq \mathcal{D}(X',C')$ containing $\ord_{E'}$. Furthermore, every valuation $v\in I$ is special over $X$ and every valuation $v'\in I'$ is special over $X'$
\item [(c)] Let $v\in I$. Let $v'\in I'$ be the image of $v$ under the homeomorphism $I\simeq I'$. Then, the $\GG_m$-equivariant degeneration of $X$ induced by $v$ and the $\GG_m$-equivariant degeneration of $X'$ induced by $v'$ are isomorphic.
\end{itemize}
\end{theorem}

\begin{remark}
{\em
In fact, there is a cone structure in the dual complex of $(X',C')$ in a neighborhood of $\ord_{E'}$. Since $X'$ admits a faithful $\GG_m$ action, there exists a $\GG_m$-equivariant birational morphism $X'\dashrightarrow Z\times \GG_m$ for some proper variety $Z$ \cite[Exercise 2.5]{Xu25}. Then we can write any $\GG_m$-invariant valuation on $X'$ as $(w, \xi)\in \text{Val}_{Z} \times N_{\RR}(\GG_m)$, where $N(\GG_m) \cong \ZZ$ is the coweight lattice of $\GG_m$ \cite[Lemma 6.17]{Xu25}. Via this identification, $\ord_{E'}$ corresponds to $(0,1)\in \text{Val}_{Z} \times N_{\RR}(\GG_m)$. 

Suppose $v' = (w, \xi)\in \text{Val}_Z\times N_{\RR}(\GG_m)$. Then there is a line of valuations \[v_{t}':= ((1-t)w, \xi + (1-\xi)t), t\in [0,1]\] connecting $v'$ and $\ord_{E'}$.
For $v'\in \text{QM}(X',C')$ such that $[v']\in \mathcal{D}(X',C')$ is sufficiently close to $\ord_{E'}$, by \cite[Proposition 6.19]{Xu25}, the entire line $\{v_t':t\in [0,1]\}$ lies inside $\text{QM}(X',C')$. This implies that around $\ord_{E'}$, $\mathcal{D}(X',C')$ is a (real) cone with vertex $\ord_{E'}$. 
}
\end{remark}

\subsection*{Sketch of the proof}

The proof of Theorem~\ref{mainthm:classification_of_special_valuations_on_del_pezzo_surfaces} and Theorem~\ref{mainthm:structure_of_the_space_of_special_valuations} is based on Proposition~\ref{prop: special_divisor_dP_surface_criterion} and Proposition~\ref{prop: special_valuation_dP_surface_criterion}, which are explicit criterion of special valuations over a surface $X$. We can use Proposition~\ref{prop: special_divisor_dP_surface_criterion} to classify special divisorial valuations over $X$. However, for special quasi-monomial valuations, we need to find a linear combination of components of $C$ whose strict transforms on various weighted blow-ups of $X$ are all ample. Bigness and nefness of such divisors can be verified by computing intersection products. Our key observation is Lemma~\ref{lemma:big_and_nef_and_not_ample_equivalent_to_existence_of_unicuspidal_curves}, which says that the obstruction of ampleness of such divisors is precisely the existence of certain unicuspidal rational curves on $X$.

The proof of Theorem~\ref{mainthm:local_structure_of_dual_complex_of_special_degenerations} is based on Construction~\ref{construction: special_degeneration}, which involves a series of MMPs. By explicitly analyzing the MMP steps, we show that an intermediate step of this construction naturally leads to a qdlt Fano type model of $X'$ (see Theorem~\ref{thm:qdlt_Fano_type_model_on_special_degeneration}). This implies the finite generation of the graded ring associated to a quasi-monomial valuation $v'$ over $X'$. Then, by a purely algebraic argument (Theorem~\ref{thm:isomorphism_of_graded_rings_for_special_degenerations}), we can lift $v'$ to a quasi-monomial valuation $v$ over $X$ whose associated graded ring is isomorphic to that of $v'$.

\subsection*{Acknowledgements} The author would like to thank his advisor Chenyang Xu for suggesting this problem, answering questions, and providing many insights into related topics. The author would also thank Zhiyuan Chen for useful discussions and comments. The author is supported by NSF Grant DMS-2201349.

\section{Preliminaries}

We work over the complex numbers $\CC$. In this section, we introduce some notations and preliminary results
regarding weighted blow-ups, dual complexes, quasi-monomial valuations, test configurations, and special degenerations of Fano varieties.

\subsection{Weighted blow-ups at a point}
In this subsection, we recall the definition of weighted blow-ups at a smooth point. The general definitions can be found in \cite{AT23}.

\begin{definition}
{\em 
Let $X = \Spec A$ be an $n$-dimensional affine variety and $x\in X$ be a smooth closed point. Let $f_1,\ldots, f_n\in \mm_x$ generate $\mm_x$. Let $w = (w_1,\ldots,w_n)\in \ZZ_{>0}^n$. Define the graded algebra $I_\bullet = \bigoplus_{m\in \NN}I_m T^m\subseteq A[T]$ as the integral closure of the image of the homomorphism
\begin{align*}
A[x_1,\ldots, x_n]&\to A[T] \\
x_i &\mapsto f_iT^{w_i}
\end{align*}
The weighted blow-up of $X$ at $x$ with weight $w$ along $f_1,\ldots,f_n$ is defined as
\[
\pi: X' = \Proj_A I_\bullet \to  X.
\]
}
\end{definition}
\begin{remark}\label{remark: weighted_blow_up}
Consider the inclusion $\OO_{X'}(1)\subseteq \OO_{X'}$ induced by $I_{m + 1}\subseteq I_m$. Let $E$ be the divisor associated to this inclusion. Then $\OO_{X'}(-E) \cong \OO_{X'}(1)$ and $\pi_*\OO_{X'}(-mE) \cong I_m$ for every $m\in \NN$. Since $\OO_{X'}(1)$ is $\pi$-ample, there exists a sufficiently divisible $r$ such that $rE$ is Cartier and $\OO_{X'}(-rE)$ is $\pi$-very ample. Then
\[
X' \cong \Proj_X\bigoplus_{m\in \NN}\pi_*\OO_{X'}(-mrE) = \Proj_X\bigoplus_{m\in \NN}I_r^{m}.\]
In particular, $\pi: X'\to X$ is the blow-up of $X$ along the ideal $I_r$.
\end{remark}

Next, we define the weighted blow-up at a point along a normal crossing divisor. Let $X$ be a $n$-dimensional variety and $D$ a normal crossing divisor on $X$ such that 
the local equation of $D$ in an analytic neighborhood of a smooth point $x\in X$ is given by $f_1f_2\cdots f_n = 0$, where $f_i\in \hat{\OO}_{X,x}$. Let $w = (w_1,\ldots,w_n)\in \ZZ_{> 0}^n$. 

\begin{definition}
{\em
Let $X$ be an $n$-dimensional variety and $x\in X$ be a smooth closed point. Let $D$ be a normal crossing divisor on $X$ such that the local equation of $D$ in an analytic neighborhood of $x$ is given by $f_1f_2\cdots f_n = 0$, where $f_i\in \hat{\OO}_{X,x}$. Let $w = (w_1,\ldots,w_n)\in \ZZ_{> 0}^n$. 

After shrinking $X$ we may assume that $X = \Spec A$ for some integral domain $A$ and $f_1,\ldots,f_n\in \hat{A}$, where $\hat{A}$ is the completion of $A$ at $\mm = \mathfrak{m}_x$. Let $\hat{\mm} = \mm \hat{A}$. Since $A/\mathfrak{m}^r \to \hat{A}/\hat{\mathfrak{m}}^r$ is an Artinian ring isomorphism for all $r \geq 1$, there is a one-to-one correspondence
\begin{align*}
    \{\mathfrak{m}\text{-primary ideals of }A\} &\leftrightarrow \{\hat{\mathfrak{m}}\text{-primary ideals of }\hat{A}\}  \\
    I &\mapsto I\cdot \hat{A} \\
    \hat{I}\cap A &\leftarrow \hat{I}
\end{align*}
By Remark~\ref{remark: weighted_blow_up}, there exists a $\hat{\mm}$-primary ideal $\hat{I}$ such that 
\[
\hat{X}':=\Proj_{\hat{A}} \bigoplus_{m\in \NN} \hat{I}^m \to \Spec \hat{A}\]
defines the weighted blow-up of $\Spec \hat{A}$ at $\hat{\mm}$ with weight $w$ along $f_1,\ldots,f_n$. Let $I = \hat{I}\cap A$. 

We define the weighted blow-up of $X$ at $x$ with weight $w$ along $D$ as
\[
X' := \Proj_A \bigoplus_{m\in \NN} I^m \to \Spec A.\]
}

\end{definition}
The following lemma is immediate from the correspondence of $\mathfrak{m}$ and $\hat{\mathfrak{m}}$-primary ideals.
\begin{lemma}\label{lemma:weighted_blow_up_nc_divisor}
    There is a Cartesian diagram \[
\begin{tikzcd}
   \hat{X}'\arrow{r}\arrow{d}{\hat{\pi}} & X'\arrow{d}{\pi} \\
   \Spec \hat{A} \arrow{r} & \Spec A
\end{tikzcd}
\]
Furthermore, we also have the following:
\begin{itemize}
    \item [(a)] The exceptional locus of $\pi$ is isomorphic to the weighted projective space $\PP(w_1,\ldots,w_n)$,
    \item [(b)] The strict transform of $D$, when restricted to the exceptional divisor $\PP(w_1,\ldots,w_n)$, is the torus invariant divisor $(w_1w_2\cdots w_n = 0)$.
\end{itemize}
\end{lemma}
\begin{proof}
This diagram is the base change of $\pi$ along $\Spec \hat{A}\to \Spec A$, by the correspondence between $\mm$-primary ideals and $\hat{\mm}$-primary ideals.\\

(a) Since $I$ is a $\mathfrak{m}$-primary ideal, $\pi$ is an isomorphism away from $x\in \Spec A$. Thus, the exceptional locus of $\pi$ is $\pi^{-1}(x) \cong \hat{\pi}^{-1}(x)\cong \PP(w_1,\ldots,w_n)$, where the second isomorphism follows from the fact that $\hat{\pi}$ is the weighted blow-up along an snc divisor.

(b) Note that $\hat{D}$, the pullback of $D$ on $\Spec \hat{A}$, is given by $(f_1f_2\cdots f_n = 0)$. Since $\hat{\pi}$ is the weighted blow-up along the snc divisor $(f_1f_2\cdots f_n = 0)$, the strict transform of $(f_i = 0)$, when restricted to the exceptional divisor, is the $i$th coordinate hyperplane in $\PP(w_1,\ldots,w_n)$.
\end{proof}

\subsection{Dual complexes}
We follow the definition of dual complexes in \cite{dFKX17}.

\begin{definition}
{\em
Let $(Y,D_Y)$ be a dlt pair. Write $\lfloor D_Y \rfloor = \sum_{i\in I} D_i$. Then the dual complex $\mathcal{D}(\lfloor D_Y\rfloor )$ is the simplicial complex obtained as follows. For each irreducible component $Z$ of $\bigcap_{j\in J} D_j$, we associate a $(|J|-1)$ dimensional simplex $v_Z$. Furthermore, if $W$ is the unique irreducible component of $\bigcap_{j\in J\setminus\{i\}} D_j$ containing $W$, then we attach $v_W$ to the corresponding face of $v_Z$.
}
\end{definition}

By \cite[Proposition 11]{dFKX17}, we may define the dual complex of a log canonical pair.
\begin{definition}
Let $(X,D)$ be a log canonical pair. Let $(Y,D_Y)$ be a dlt modification of $(X,D)$. Then we define $\mathcal{D}(X,D)$ to be the PL-homeomorphism type of $\mathcal{D}(\lfloor D_Y\rfloor)$.
\end{definition}

\subsection{Test configurations}

\begin{definition}
{\em
Let $X$ be an $n$-dimensional projective variety and $L$ be an ample $\QQ$-Cartier divisor on $X$. Let $r\in \mathbb{N}$ such that $rL$ is a very ample line bundle. A \textit{test configuration $(\mathcal{X}, \mathcal{L}_r)$ of $(X,L)$ with index $r$} is given by the following data:
\begin{itemize}
    \item $\mathcal{X}$ is a normal scheme with a $\GG_m$-action,
    \item $\pi: \mathcal{X}\to \AA^1$ is a $\GG_m$-equivariant flat morphism, where $\GG_m$ acts on $\AA^1$ by the multiplication $(t,a)\mapsto ta$, such that for $t\neq 0$, there is an isomorphism $\phi_t: \mathcal{X}_t\cong X$ where $\mathcal{X}_t$ is the fiber of $\mathcal{X}$ over $t\in \AA^1$.
    \item a $\GG_m$-linearized very ample line bundle $\mathcal{L}_r$ on $\mathcal{X}$ such that for $t\neq 0$, we have an isomorphism
    \[
    \mathcal{L}_r|_{\mathcal{X}_t} \cong \phi_t^* \OO_X(rL).
    \]
\end{itemize}

We say that $(\mathcal{X}, \mathcal{L}_r)$ is a trivial test configuration if $\mathcal{X}$ is the constant family $X\times\AA^1$ over $\AA^1$ with the trivial $\GG_m$ action and $\mathcal{L}_r$ is the pullback of $\OO_X(rL)$ along the projection $X\times \AA^1\to X$.
}
\end{definition}

\begin{remark}
{\em
For any test configuration $(\mathcal{X}, \mathcal{L}_r)$ (of index $r$), there is an isomorphism
\[
\phi: \mathcal{X} \times_{\AA^1} \GG_m \cong X\times \GG_m
\]
such that the restriction of $\phi$ on $\mathcal{X}_t$ is $\phi_t$ when $t\neq 0$. Furthermore, the restriction of $\mathcal{L}_r$ on $\mathcal{X}\times_{\AA^1}\GG_m$ is isomorphic to $\phi^*(\OO_X(rL)\boxtimes \OO_{\GG_m})$.
}
\end{remark}

\begin{definition}
{\em
Let $X$ be a klt Fano variety. Let $(\mathcal{X}, \mathcal{L}_r)$ be a test configuration of $(X,-K_X)$ with index $r$. Then $(\mathcal{X}, \mathcal{L}_r)$ is called weakly special (resp. special) if $(\mathcal{X}, \mathcal{X}_0)$ is log canonical (resp. plt) and $\frac{1}{r}\mathcal{L}_r\sim_{\QQ} -K_{\mathcal{X}}$.
}
\end{definition}
Suppose $\mathcal{X}$ is a non-trivial weakly special test configuration such that $\mathcal{X}_0$ is integral. Then $\ord_{\mathcal{X}_0}$ is a divisorial valuation on $K(\mathcal{X}) \cong K(X\times \AA^1) \cong K(X)(t)$. The restriction of $\ord_{\mathcal{X}_0}$ on $K(X)\subseteq K(X)(t)$ is a non-trivial divisorial valuation $c\cdot \ord_E$ on $K(X)$. The following Proposition says that one can reconstruct $\mathcal{X}$ from this divisorial valuation $c\cdot \ord_E$.

\begin{proposition}[\text{\cite[Lemma 4.18]{Xu25}}]\label{prop:rees_construction}
Let $X$ be a klt Fano variety. Let $(\mathcal{X}, \mathcal{L}_r)$ be a non-trivial weakly special test configuration of $(X,-K_X)$ with an integral central fiber $\mathcal{X}_0$. Then, the restriction of $\ord_{\mathcal{X}_0}$ on $K(X)$ is a non-trivial divisorial valuation $c\cdot \ord_E$ on $X$, for some divisor $E$ over $X$ and $c>0$. Then
\[\mathcal{X} \cong \Proj_{\AA^1} \left(\bigoplus_{m\in r\NN} \bigoplus_{s\in \ZZ}\mathcal{F}_E^{s} R_m t^{-s}\right),
\]
where $R_m = H^0(X, -mK_X)$ and $\mathcal{F}_E^\bullet R_m$ is the decreasing, multiplicative filtration on $R_m$ given by
\[
\mathcal{F}_E^s R_m = \{f\in R_m: c\cdot \ord_E(f)\geq s\}.
\]
Furthermore, the central fiber of $\mathcal{X}$ is
\[
\mathcal{X}_0 \cong \Proj\left(\bigoplus_{m\in r\NN} \gr_E^\bullet R_m\right), \text{ where }\gr_E^s R_m = \frac{\mathcal{F}_E^s R_m}{\mathcal{F}_E^{>s} R_m}.
\]
\end{proposition}

\begin{definition}
{\em 
Let $X$ be a klt Fano variety. A divisor $E$ over $X$ (or equivalently, a divisorial valuation $\ord_E$ over $X$) is called special (resp. weakly special) if there is a non-trivial special test configuration (resp. weakly special test configuration with integral fibers) $\mathcal{X}$ of $X$, such that the restriction of $\ord_{\mathcal{X}_0}$ on $K(X)$ is of the form $c\cdot \ord_E$ for some constant $c > 0$.

Equivalently, by Proposition~\ref{prop:rees_construction}, a divisor $E$ is special (resp. weakly special) if
\[
\mathcal{X} = \Proj_{\AA^1} \left(\bigoplus_{m\in r\NN} \bigoplus_{s\in \ZZ}\mathcal{F}_E^{s} R_m t^{-s}\right)
\]
defines a special test configuration (resp. weakly special configuration with integral fibers) of $(X,-K_X)$ for some choice of $\mathcal{L}_r$. In this case, we say that $\mathcal{X}$ is the special (resp. weakly special) test configuration induced by $E$.
}
\end{definition}

The following theorem is a criterion for weakly special divisors.
\begin{theorem}[\text{\cite{BLX22}}]\label{thm:weakly_special_divisor_criterion}
Let $X$ be a klt Fano variety and $E$ be a $\QQ$-Cartier integral divisor over $X$. Then $E$ is weakly special if and only if there exists a $\QQ$-complement $\Delta$ of $X$ such that $E$ is an lc place of $(X,\Delta)$. 
\end{theorem}

The key ingredient of the proof of Theorem~\ref{thm:weakly_special_divisor_criterion} is the construction of a weakly special test configuration $\mathcal{X}\to \AA^1$ induced by a log canonical place $E$.
\begin{construction}[\text{\cite[Theorem 4.22]{Xu25}}]\label{construction: special_degeneration}
{\em Suppose that $E$ is a weakly special divisor over a $\QQ$-factorial klt Fano variety $X$. Then we may construct a weakly special test configuration $\mathcal{X}\to \AA^1$ induced by $E$ in four steps, as illustrated by the following diagram:
\[
    X\times \AA^1 \overset{(1)}\leftarrow \mathcal{Y}  \overset{(2)}\dashrightarrow \mathcal{X}'\overset{(3)}\dashrightarrow \mathcal{X}^m \overset{(4)}\to \mathcal{X}.
\]
(1) Let $\mathcal{Y} \to X\times \AA^1$ extract the divisor $S_1$ corresponding to the valuation $(\ord_E, \ord_s)$ on $X\times \AA^1_s$. Then we may write the central fiber of $\mathcal{Y}$ as $X_1 \cup S_1$, where $X_1$ is the strict transform of $X\times \{0\} \subseteq X\times \AA^1$. Denote $\Delta_{\mathcal{Y}}$ be the strict transform of $\Delta\times \AA^1$.\\
(2) We run a $\GG_m$-equivariant MMP over $\AA^1$ on $\mathcal{Y}$ with the divisor $K_\mathcal{Y} + \Delta_{\mathcal{Y}} + X_1 + (1-\epsilon)S_1 \sim_{\QQ} -\epsilon S_1$ for a sufficiently small $\epsilon >0$. This gives us a model  $\mathcal{X}'\to \AA^1$ with an irreducible central fiber $S_2$, because the strict transform of $X_1$ is contracted in this MMP.\\
(3) We run a $\GG_m$-equivariant $-K_{\mathcal{X}'}$-MMP over $\AA^1$. More precisely, take $G\sim_{\QQ} -K_X$ a general ample divisor on $X$ and let $\mathcal{G}_{\mathcal{X}'}$ be the closure of $G\times (\AA^1-\{0\})$ on $\mathcal{X}'$. Then we run a MMP over $\AA^1$ with the divisor $K_{\mathcal{X}'} + \Delta_{\mathcal{X}'} +  \epsilon \mathcal{G}_{\mathcal{X}'}\sim_{\QQ, \AA^1} -\epsilon K_{\mathcal{X}'}$ for a sufficiently small $\epsilon > 0$. The result of this MMP is a relative minimal model $\mathcal{X}^{m} \to \AA^1$. \\
(4) We take the relative ample model of $-K_{\mathcal{X}^{m}}$ over $\AA^1$ to obtain $\mathcal{X}$.\\

We note that after restricting to the open subset $\AA^1-\{0\}\subseteq \AA^1$, each birational model $\mathcal{Y}, \mathcal{X}', \mathcal{X}^m, \mathcal{X}$ is isomorphic to the trivial family $X\times (\AA^1-\{0\})$. Modifications only happen in the central fibers.
}
\end{construction}

The following theorem is a criterion for  special divisors over $X$.
\begin{theorem}[\cite{LXZ22}]\label{thm:special_divisor_criterion}
    Let $X$ be a klt Fano variety and $E$ be a $\QQ$-Cartier integral divisor on $X$. Then the following are equivalent:
    \begin{itemize}
        \item [(1)] $E$ is special.
        \item [(2)] $E$ is weakly special and for any $\QQ$-Cartier $\QQ$-divisor $D$ on $X$, there exists $\epsilon > 0$ and an effective $\QQ$-divisor $D'\sim_{\QQ}-K_X - \epsilon D$, such that $(X, D' + \epsilon D)$ is log canonical with $E$ as an lc place.
        \item [(3)] $A_X(E) < T_X(E)$ and there exists a $\QQ$-complement $\Delta$ of $X$ such that $E$ is the only lc place of $(X,\Delta)$.
        \item [(4)] There exists a a divisor $D\sim_{\QQ}-K_X$ and $t\in (0,1)$ such that $(X, tD)$ is lc and $E$ is the only lc place of $(X, tD)$.
        \item [(5)] There exists a projective birational morphism $\mu: Y\to X$ and an effective $\QQ$-divisor $D_Y$ on $Y$ such that $(Y, E+D_Y)$ is plt and $-(K_Y + E + D_Y)$ is ample.
    \end{itemize}
\end{theorem}

\subsection{Quasi-monomial valuations}
\begin{definition}
{\em
Let $(Y,E)$ be a dlt pair. Suppose $E_1,\ldots,E_r$ are irreducible components of $\lfloor E\rfloor$ such that $\bigcap_{i=1}^r E_i\neq \varnothing$. Let $C$ be a component of $\bigcap_{i=1}^r E_i$ and $\eta$ be the generic point of $C$. Around $\eta$, $E_i$ is given by an equation $y_i\in \OO_{Y,\eta}$. Let $\alpha = (\alpha_1,\ldots,\alpha_r) \in \RR_{\geq 0}^r$. For each $f\in \OO_{Y,\eta}$, we write
\[
f = \sum_{\beta\in \mathbb{N}^r}c_{\beta} y^\beta, \text{ where }y^\beta = \prod_{i=1}^r y_i^{\beta_i}.
\]
We define a valuation $v_\alpha$ on $\OO_{Y,\eta}$ by setting
\[
v_{\alpha}(f) = \min_{\beta\in \mathbb{N}^r} \left\{\sum_{i=1}^r \alpha_i\beta_i: c_\beta \neq 0\right\}.
\]
Then, $v_{\alpha}$ naturally extends to a valuation on $Y$. We define
\[
\text{QM}_\eta(Y, E) = \{v_\alpha: \alpha\in \mathbb{R}_{\geq 0}^r\}
\]
and
\[
\text{QM}(Y,E) = \bigcup_{\eta'} \text{QM}_{\eta'}(Y,E)
\]
where $\eta'$ runs through the generic point of every stratum of $\lfloor E\rfloor$.
}
\end{definition}

\begin{remark}\label{remark:dual_complex_qm_vals}
{\em
There is a natural $\RR_{>0}$ action on $\text{QM}(Y,E)$ which rescales any valuation by a positive constant.
Then, the dual complex $\mathcal{D}(Y,E)$ can be identified with the quotient of $\text{QM}(Y,E)\setminus \{0\}$ by the $\RR_{>0}$-action on $\text{QM}(Y,E)$.
}
\end{remark}

\begin{definition}
{\em
    Let $X$ be a variety. We say that $v$ is a quasi-monomial valuation over $X$ if there exists a dlt model $ (Y,E)\to X$, such that $v\in \text{QM}(Y,E).$

    Let $(X,D)$ be a log canonical pair. We say that $v\in \text{QM}(X,D)$ if there exists a dlt modification $ (Y,E)\to (X,D)$ such that $v\in \text{QM}(Y,E)$.
    }
\end{definition}

\begin{definition}
{\em
    Let $X$ be a normal projective variety. We say that a projective birational morphism $\mu: (Y, E)\to X$ is a qdlt Fano type model if there exists an effective $\QQ$-divisor $D$ on $Y$ such that $(Y, E+D)$ is qdlt (see \cite[Definition 35]{dFKX17} for the precise definition), $\lfloor E + D \rfloor = E$, and $-(K_Y+E+D)$ is ample.

    We say that a quasi-monomial valuation $v$ over $X$ is special, if there exists a qdlt Fano type model $(Y,E)\to X$ such that $v\in \text{QM}(Y,E)$.
    }
\end{definition} 

We also need the following theorems about the finite generation of the graded ring associated to a quasi-monomial valuation on a qdlt Fano type model.
\begin{theorem}[\text{\cite{XZ25}},\text{\cite[Section 5.2]{Xu25}}]\label{thm:qdlt_fano_type_model_and_higher_rank_degeneration}
Let $(Y,E)\to X$ be a qdlt Fano type model of a klt Fano variety $X$. Suppose $E = E_1 + \cdots + E_p$ is a reduced divisor and $\bigcap_{i=1}^p E_i \neq \varnothing$. Let $\eta$ be the generic point of an irreducible component of $\bigcap_{i=1}^p E_i$. Let $v\in \text{QM}(Y,E)$ be a quasi-monomial valuation whose center on $Y$ contains $\eta$. Denote $R_m = H^0(X, -mK_X)$ for $m\in r\NN$, where $r$ is the Cartier index of $K_X$. Then the following statements hold:
\begin{itemize}
\item $\displaystyle{\bigoplus_{m\in r\NN}\gr_v^\bullet R_m}$ is finitely generated.
\item $\displaystyle{\Proj\left(\bigoplus_{m\in r\NN}\gr_v^\bullet R_m\right)}$ is a klt Fano variety.
\item If the center of $v$ is $\eta$, then $\displaystyle{\bigoplus_{m\in r\NN}\gr_v^\bullet R_m}$ is isomorphic to a fixed graded ring.
\end{itemize}
\end{theorem}

\begin{theorem}[\text{\cite[Theorem 5.19]{Xu25}}]\label{thm:rational_space_containing_special_valuation_is_special}
Let $X$ be a klt Fano variety. Denote $$\displaystyle{R = \bigoplus_{m\in r\NN} H^0(X,-mK_X)}$$ be the section ring of the anti-canonical divisor. Suppose $v\in \text{QM}_\eta(Y,E)$ for some dlt model $(Y,E)\to X$ such that $\gr_v^\bullet R$ is finitely generated. Then, there exists an open neighborhood $U$ of $v$ in the minimal rational vector space containing $v$, such that for all $w\in U$,
\[
\gr_w^\bullet R\cong \gr_v^\bullet R.
\]
\end{theorem}

\subsection{Unicuspidal curves}
\begin{definition}
{\em
Let $X$ be a smooth projective surface and $C\subseteq X$ be a nodal curve with a node $x$. Let $x_1, x_2$ denote the analytic coordinate near $x$ such that the local equation of $C$ is given by $x_1x_2 = 0$. We say that an integral curve $D$ on $X$ is a $(p,q)$-unicuspidal curve well-formed with respect to $(C,x)$ if 
\begin{itemize}
\item $D$ has a unique singular point $x$, which is a $(p,q)$-cusp,
\item the Newton polygon of the local equation of $D$ in the $x_1,x_2$ coordinate is the region 
\[
\left\{(x,y)\in \RR^2_{\geq 0} :\frac{x}{p} + \frac{y}{q} \geq 1\right\}.\]
\end{itemize}
}
\end{definition}

\begin{proposition}
Suppose $X$ is a smooth del Pezzo surface and $C$ is an anti-canonical integral divisor on $X$ with nodal singularities. Let $x\in C$ be a node and fix an order of the two (analytic) branches of $C$ at $x$. For $t>0$, let $v_t\in \text{QM}(X,C)$ denote the quasi-monomial valuation centered at $x$ with weight $(1,t)$ along the two branches of $C$. Let $p,q$ be coprime positive integers such that $pq\vol(-K_X) > (p+q)^2$. Then the following are equivalent:
    \begin{itemize}
        \item [(a)] There exists a line bundle $L$ such that $p+q = (-K_X\cdot L)$ and $pq = (L^2) + 1$.
        \item [(b)] There exists a $(p,q)$-unicuspidal rational curve $D$ well-formed with respect to $C$ such that $D\cap C = \{x\}$.
        \item [(c)] Suppose $f: Y\to X$ is the $(p,q)$-weighted blow-up along the two branches of $C$. Then $C_Y:= f_*^{-1}C$ is big and nef but not ample. 
    \end{itemize}
\end{proposition}

\begin{proof}
We first show that (a) implies (b). By Kawamata-Viehweg vanishing and Riemann-Roch,
\[
h^0(L) = \chi(L) = \chi(\OO_X) + \frac{1}{2}(L\cdot L-K_X) =  1 + \frac{1}{2}(L^2) + \frac{1}{2}(-K_X\cdot L).
\]
Thus, condition (ii) in (a) implies that
\[
h^0(L) = \frac{1}{2}(pq+p+q+1).
\]

Let $v= \ord_E$, where $E = \text{Ex}(f)$ the exceptional divisor. By Pick's theorem,
\[
\text{colength}(\mathfrak{a}_{pq}(v)) = \frac{1}{2}(pq+p+q-1) < h^0(L).
\]
Thus, there is a divisor $D\in |L|$ with $v(D) \geq pq$.

Suppose that $D$ is not integral, then there exists $0\leq D'< D$ such that 
\[
\frac{v(D')}{\deg(D')} \geq \frac{v(D)}{\deg (D)} \implies v(D')\geq \frac{(-K_X\cdot D')}{(-K_X\cdot D)}pq = (-K_X\cdot D') \frac{pq}{p+q},
\]
where the last equality is given by (i).
If $D' \neq C$, We compute the local intersection numbers to get the following inequality:
\[
(-K_X\cdot D') \geq (-K_X\cdot D')_x \geq v(D')(\frac{1}{p}+\frac{1}{q})\geq \frac{(-K_X\cdot D')}{(-K_X\cdot D)} (p+q) = (-K_X\cdot D').
\]
Thus, all the above equalities must hold. In particular,
\[
(-K_X\cdot D') \frac{pq}{p+q} = v(D')
\]
must be an integer. However, since $\gcd(p,q) = 1$, $(-K_X\cdot D')$ divides $p+q$ and hence must be at least $p+q$. This contradicts with $D' < D$.

If $D' = C$, then
\[
\frac{v(C)}{\deg C} \geq \frac{v(D)}{\deg D} \implies \frac{p+q}{\vol(-K_X)} \geq \frac{pq}{(-K_X\cdot L)} \implies (p+q)^2 \geq pq\vol(-K_X),
\]
which contradicts with (iii). 

Since $D$ is integral, the newton polygon of $D$ in $x_1,x_2$ coordinates passes through $(x_D,0)$ and $(0,y_D)$ for some $x_D, y_D>0$. Then $x_D\geq q$ and $y_D\geq p$. Again, by computing local intersection numbers, we get
\[
p+q = (-K_X\cdot D) \geq (-K_X\cdot D)_x = x_D + y_D
\]
Therefore, $x_D = q$, $y_D = p$ and hence $v(D) = pq$. Thus, the newton polygon of $D$ cannot have any vertices below the line connecting $(q,0)$ and $(0,p)$. The equality $(C\cdot D) = (C\cdot D)_x$ also shows that $C\cap D = \{x\}$.

Since $D$ already has a $(p,q)$-cusp,
\[
2g(D) - 2 \leq (K_X+D)\cdot D - (p-1)(q-1) = -2
\]
Thus, $g(D) = 0$ and $D$ does not have any other singularities.

Next, we show that (b) implies (c). Let $D$ be the $(p,q)$-unicuspidal rational curve in (a). Then $D_Y:= f_*^{-1}(D)$ is a smooth rational curve on $Y$. Furthermore, if $u_0$ and $u_1$ are the coordinates on $E = \text{Ex}(f) \cong \PP^1$, then $D_Y\cap E = (u_0^q - u_1^p = 0)$ and $C_Y\cap E = (u_0u_1 = 0)$. As a result, $D_Y\cap C_Y = \varnothing$ because they do not intersect outside of $E$. Thus, $C_Y$ is not ample. Furthermore, we may compute
\[
(C_Y^2) = (C^2) - (p+q)^2(E^2) = \vol(-K_X) - \frac{(p+q)^2}{pq} > 0.
\]
Hence, $C_Y$ is big and nef.

Finally, we prove that (c) implies (a). Suppose that $C_Y$ is big and nef but not ample, then there exists an irreducible curve $D_Y$ on $Y$ which does not intersect $C_Y$. 

We claim that $(D_Y^2) < 0$. If $(D_Y^2) \geq 0$, then $D_Y$ is nef and hence semiample. Since $D_Y$ is not numerically trivial, the Iitaka dimension of $D_Y$ is positive. Thus, $|mD_Y|$ contains a nontrivial pencil of divisors which dominate $Y$. However, since $(C_Y\cdot D_Y) = 0$, $C_Y$ cannot intersect any element of $|mD_Y|$. This contradicts with the bigness of $C_Y$.

Let $D = f(D_Y)$. We can compute  
\[
(C\cdot D) = (C_Y\cdot D_Y) + \frac{p+q}{pq}\ord_{E}(D) = \frac{p+q}{pq}\ord_{E}(D).
\]
Since $\gcd(p,q) = 1$, $\ord_{E}(D) = kpq$ for some integer $k\geq 1$. Furthermore,
\begin{align*}
    -2 &\leq (K_{Y} + D_Y)\cdot D_Y = (f^*K_{X} + (p+q-1)E + D_Y)\cdot D_Y \\
    &= (-C_Y - E + D_Y)\cdot D_Y \\
    &=  -E\cdot (f^*D - \ord_{E}(D)E) + (D_Y^2) \\
    &= -k + (D_Y^2)
\end{align*}
Thus, $k = 1$ and $(D_Y^2) \geq -1$. Since $(D_Y^2) = (D^2) - k^2pq$ is an integer, $(D_Y^2) = -1$. This gives 
\[
(C\cdot D) = p+q, \quad (D^2) = pq-1,
\]
as desired.
\end{proof}

\section{Special Valuations over Log Calabi-Yau Surface Pairs}
Let $X$ be a smooth del Pezzo surface and $C\sim -K_X$ be an effective divisor on $X$ such that $(X,C)$ is log canonical.
The goal of this section is to answer the following question:
\begin{question}
Let $v\in \text{QM}(X,C)$ be a quasi-monomial valuation over $X$. When is $v$ special over $X$? 
\end{question}

By adjunction, $C$ is one of the following:
\begin{itemize}
\item[(i)] $C$ is an irreducible smooth curve.
\item[(ii)] $C$ is an irreducible nodal rational curve with a unique node.
\item[(iii)] $C$ consists of a circle of irreducible smooth rational curves.
\end{itemize}

In case (i), $\text{QM}(X,C) = \{\ord_C\}$ and $\ord_C$ is not a special valuation by Theorem~\ref{thm:special_divisor_criterion}. For the rest of this section, we focus on case (ii) and (iii).

\subsection{Numerical criterion for special valuations}
\begin{proposition}\label{prop: special_divisor_dP_surface_criterion}
Let $X$ be a smooth Fano surface and $C\sim -K_X$ be a nodal curve on $X$ such that $(X,C)$ is log canonical. Let $E$ be an log canonical place of $(X,C)$. Let $f: Y\to X$ be the projective birational morphism extracting the divisor $E$. If $f$ is the identity map (i.e., $E$ is a divisor contained in $X$), let $C_Y = C - E$. Otherwise, let $C_Y$ be the strict transform of $C$ on $Y$. Then the following statements are equivalent:
\begin{itemize}
\item[(a)] $E$ is a special divisor over $X$.
\item[(b)] There exists an effective $\QQ$-divisor $B$ on $Y$ such that $\supp(B)= \supp(C_Y)$ and $B\cdot F > 0$ for any component $F$ of $\supp(C_Y)$.
\end{itemize}
\end{proposition}
\begin{proof}
We first show that (a) implies (b). Since $E$ is a special divisor over $X$, by Theorem~\ref{thm:special_divisor_criterion}, there exists an effective $\QQ$-divisor $D$ such that $\lfloor D+E\rfloor = E$, $(Y, D+E)$ is plt, and $A:= -(K_Y + D + E)$ is ample. Since
\[
0\sim f^*(K_X+C) = K_Y + E + C_Y,\]
we have $C_Y \sim_{\QQ} A + D$. Write $D = D_1 + D_2$ where $\supp(D_1)\subseteq \supp(C_Y)$ and no component of $D_2$ is contained in $\supp(C_Y)$. Define $B = C_Y - D_1\sim_{\QQ} A+ D_2$. Then $B$ is effective and $\supp(B) = C_Y$. Furthermore, for any irreducible component $F$ of $\supp(C_Y)$,
\[
F\cdot B = F\cdot A + F\cdot D_2 > 0.\]

Next, we prove that (b) implies (a). Let $B$ the the effective $\QQ$-divisor on $Y$ satisfying (b). After rescaling we may assume that $\lfloor B\rfloor =0$. Since $B^2 > 0$, $B$ is nef and hence semiample. Thus, there exists an effective $\QQ$-divisor $F$ such that $B\cdot F = 0$ and $B-\epsilon F$ is ample for any sufficiently small $\epsilon > 0$. In particular, $E\not\subseteq\supp(F)$. Define $D = C_Y - B + \epsilon F\geq 0$. Then for sufficiently small $\epsilon > 0$, $\lfloor D+E\rfloor = E$, $(Y, D+E)$ is plt, and $-(K_Y + D + E) \sim_{\QQ} B-\epsilon F$ is ample. By Theorem~\ref{thm:special_divisor_criterion}, $E$ is a special divisor over $X$.
\end{proof}

\begin{proposition}\label{prop: special_valuation_dP_surface_criterion}
Let $X$ be a smooth Fano surface and $C\sim -K_X$ be a nodal curve on $X$ such that $(X,C)$ is log canonical. Let $x\in C$ be a node. For $t>0$, let $v_t$ denote the quasi-monomial valuation centered at $x$ with weight $(1,t)$ along the two branches of $C$. For $t\in \QQ$, let $E_t$ denote the divisor over $X$ corresponding to the divisorial valuation $v_t$, and let $Y_t\to X$ be the birational morphism extracting the divisor $E_t$. Write $C = C_1 + \cdots + C_k$ where $C_i$ is irreducible for $1\leq i\leq k$. Let $C_{i,t}$ be the strict transform of $C_i$ on $Y_t$.

Suppose that there exists an open interval $I$ and $a_1,\ldots, a_k>0$ such that $\sum_i a_i C_{i,t}$ is ample for all $t\in I\cap \QQ$. Then for all $t\in I$, $v_t$ is a special valuation over $X$ and the special degenerations of $X$ induced by all such $v_t$'s are isomorphic to each other.
\end{proposition}

\begin{proof}
Let $t_0\in I$ and $t_1, t_2\in I\cap \QQ$ such that $t_1 < t_0 < t_2$. Let $f: Y\to X$ be the birational morphism which extracts $E_{t_1}$ and $E_{t_2}$ via weighted blow-ups.\\
\noindent\textit{Step 1.} In this step, we prove the following statement: for any $t_1,t_2\in \QQ$ sufficiently close to $t_0$, $\sum_i a_i C_{i,Y}$ is big and nef. Here, $C_{i,Y}$ is the strict transform of $C_i$ on $Y$.

Write $t_1 = \frac{q_1}{p_1}$ and $t_2 = \frac{q_2}{p_2}$. By a toric computation, we have 
\begin{align*}
(E_{t_1}^2) &= -\frac{q_2}{q_1(q_2p_1-p_2q_1)}, \\
(E_{t_2}^2) &= -\frac{p_1}{p_2(q_2p_1-p_2q_1)}, \\
(E_{t_1}\cdot E_{t_2}) &= \frac{1}{q_2p_1-p_2q_1}.
\end{align*}
\begin{itemize}
\item Case 1: $C = C_1$ is irreducible. 

In this case, we have
    \[
    f^*C = C_Y + (p_1+q_1)E_{t_1} + (p_2+q_2)E_{t_2},
    \]
so $(C_Y \cdot E_{t_1}) = \frac{1}{q_1}, (C_Y \cdot E_{t_2}) = \frac{1}{p_2}$,
    and 
    \[
    (C_Y^2) = (C^2) - \frac{p_1p_2 + q_1q_2 + 2p_2q_1}{p_2q_1} = d - \frac{1}{t_1} - t_2 - 2.
    \]
For any $t\in I\cap \QQ$, we have
\[
(C_t^2) = d - \frac{(1+t)^2}{t} > 0.\]
Since the set $\{t>0:  d - \frac{(1+t)^2}{t} > 0\}$ is an open interval, we have
\[
d - \frac{(1+t_0)^2}{t_0} = d -\frac{1}{t_0} - t_0 - 2 > 0\]
Thus, $(C_Y^2)>0$ for $t_1,t_2$ sufficiently close to $t_0$ and hence $C_Y$ is big and nef.
\item Case 2: $C$ is reducible.

In this case, we may assume $x\in C_1\cap C_2$. Then
\begin{align*}
f^*C_1 &= C_{1,Y} + p_1 E_{t_1} + p_2 E_{t_2}, \\
f^*C_2 &= C_{2,Y} + q_1 E_{t_1} + q_2E_{t_2},\\
f^*C_j &= C_{j,Y}, \text{ if }j\geq 3.
\end{align*}
Thus,
\begin{align*}
\left(\sum_i a_i C_{i,Y} \right)\cdot C_{1,Y} &= \left(\sum_i a_i C_i\right) \cdot C_1  - \frac{a_1}{t_1} - a_2, \\
\left(\sum_i a_i C_{i,Y} \right)\cdot C_{2,Y} &= \left(\sum_i a_i C_i\right) \cdot C_2 - a_1 - a_2t_2, \\
\left(\sum_i a_i C_{i,Y} \right)\cdot C_{j,Y} &= \left(\sum_i a_i C_i\right) \cdot C_j, \text{ if }j\geq 3.
\end{align*}
On the other hand, for any $t\in I\cap \QQ$, 
\begin{align*}
0<\left(\sum_i a_i C_{i,t} \right)\cdot C_{1,t} &= \left(\sum_i a_i C_i\right) \cdot C_1  - \frac{a_1}{t} - a_2, \\
0<\left(\sum_i a_i C_{i,t} \right)\cdot C_{2,t} &= \left(\sum_i a_i C_i\right) \cdot C_2 - a_1 - a_2t , \\
0<\left(\sum_i a_i C_{i,t} \right)\cdot C_{j,t} &= \left(\sum_i a_i C_i\right) \cdot C_j , \text{ if }j\geq 3.
\end{align*}
Thus, $\sum_i a_i C_{i,Y}$ is big and nef for all $t_1,t_2\in I\cap \QQ$. 
\end{itemize}
From now on, we fix a choice of $t_1,t_2\in I\cap \QQ$ such that $\sum_i a_iC_{i,Y}$ is big and nef.\\

\noindent\textit{Step 2.} In this step, we prove that if there exists an irreducible curve $D_Y$ on $Y$ such that $C_{i,Y} \cdot D_Y =0$ for all $1\leq i\leq k$, then $y\not\in D_Y$, where $y$ is the unique intersection point of $E_{t_1}$ and $E_{t_2}$.

Assume that $y\in D_Y$. By the Hodge index theorem, $(D_Y^2)<0$. We may perform a sequence of blow-ups centered at $y$ to obtain $g: Z\to Y$ which satisfies the following properties:
    \begin{itemize}
        \item Every exceptional divisor of $g$ has the form $E_{t}$ for some $t\in (t_1,t_2)\cap \QQ$,
        \item Every component of $E_{t_1} + E_{t_2} + \text{Ex}(g)$ is either disjoint from $D_Z:= g^{-1}_*D_Y$, or intersects $D_Z$ at smooth points of $Z$.
        \item There exists $t'\in [t_1,t_2]\cap \QQ$ such that $E_{t'}$ and $D_Z$ intersect at some smooth points of $Z$.
    \end{itemize}
    Furthermore, we have
    \[
    g^*\left(K_Y + \sum_i C_{i,Y} + E_{t_1} + E_{t_2}\right) = K_Z + \sum_i C_{i,Z} + E_{t_1} + E_{t_2} + \text{Ex}(g)
    \]
    and hence
    \begin{align*}
        -2\leq (K_Z + D_Z)\cdot D_Z = (D_Z^2) - (E_{t_1} + E_{t_2} + \text{Ex}(g))\cdot D_Z 
    \end{align*}
    This implies that
    \[
    (E_{t_1} + E_{t_2} + \text{Ex}(g))\cdot D_Z < 2
    \]
    Note that for every component $F$ of $E_{t_1} + E_{t_2} + \text{Ex}(g)$, $(F\cdot D_Z)$ is a nonnegative integer. As a result, $E_{t'}$ is the only component which intersects $D_Z$. After blowing down every component of $E_{t_1} + E_{t_2} + \text{Ex}(g)$ except $E_{t'}$, we obtain the variety $Y_{t'}$. On $Y_{t'}$, we have $(C_{i,t'}\cdot D_{t'}) = 0$ for all $1\leq i\leq k$, where $D_{t'}$ is the image of $D_Y$. However, this contradicts with the assumption that $\sum_i a_iC_{i,t'}$ is ample.\\
    
    \noindent\textit{Step 3.} 
    By step 1 and 2, there exists an effective $\QQ$-divisor $F_Y$ on $Y$ not containing $y$, such that $\sum_i a_iC_{i,Y} - \epsilon F_Y$ is ample for any sufficiently small $\epsilon >0$. After rescaling we may assume that $a_i\in (0,1)$ for all $i$. Then
    \[
    - \left(K_Y + \epsilon F_Y + \sum_i (1-a_i)C_{i,Y}+ E_{t_1} + E_{t_2}\right) \sim_\QQ \sum_i a_iC_{i,Y} - \epsilon F_Y
    \]
    is ample. Since $y\not\in \supp(F_Y)$,  $(Y, \epsilon F_Y+\sum_i (1-a_i)C_{i,Y}+ E_{t_1} + E_{t_2} )$ is qdlt for sufficiently small $\epsilon$. This implies that $(Y, E_{t_1} + E_{t_2})\to X$ is a qdlt Fano type model. By Theorem~\ref{thm:qdlt_fano_type_model_and_higher_rank_degeneration}, for all $t\in (t_1,t_2)$, $v_t$ is special and induces isomorphic special degenerations of $X$.
\end{proof}

\subsection{Special divisors over $X$}
The goal of this subsection is to classify special divisorial valuations in $\text{QM}(X,C)$.

We first study special divisors in $\text{QM}(X,C)$ in the case where every component of $C$ is nef.
\begin{proposition}\label{cor:special_divisors_all_boundaries}
    Let $X$ be a smooth del Pezzo surface of degree $d$. Let $C\sim -K_X$ be normal crossing divisor on $X$. Let $x\in C$ be a node. For $t>0$, let $v_t$ denote the quasi-monomial valuation centered at $x$ with weight $(1,t)$ along the two branches of $C$. For $t\in \QQ_{>0}$, let $E_t$ denote the divisor over $X$ corresponding to the divisorial valuation $v_t$, and let $Y_t\to X$ be the birational morphism extracting the divisor $E_t$. Assume that $C = C_1 + \cdots + C_k$ has $k$ irreducible components and $C_i^2 \geq 0$ for all $1\leq i\leq k$.
    \begin{itemize}
        \item[(a)] Suppose $k=1$. Then $E_t$ is special over $X$ if and only if $d\geq 5$ and $t\in I_d:= \left(\frac{d-2-\sqrt{d^2-4d}}{2}, \frac{d-2+\sqrt{d^2-4d}}{2}\right)$.
        \item[(b)] Suppose $k=2$. Then $E_t$ is special over $X$ if and only if $d\geq 5$.
        \item[(c)] Suppose $k\geq 3$. Then $E_t$ is special over $X$ for all $t\in \QQ_{>0}$.
    \end{itemize}
\end{proposition}
\begin{proof}
(a) By Proposition~\ref{prop: special_divisor_dP_surface_criterion}, $E_t$ is special over $X$ if and only if $C_t$ is big and nef. We may compute
\[
C_t^2 = C^2 - \frac{(1+t)^2}{t} = d -  \frac{(1+t)^2}{t}.\]
Thus, $C_t^2 > 0$ if and only if $d\geq 5$ and $t\in I_d$.\\

(b) When $k=2$, $C_1\cdot C_2 = 2$ and $C_1^2 + C_2^2 = d-4$. For any $a_1,a_2 > 0$, we can compute
\begin{align*}
(a_1C_{1,t} + a_2C_{2,t})\cdot C_{1,t} &= a_1\left(C_1^2 - \frac{1}{t}\right) + a_2, \\
(a_1C_{1,t} + a_2C_{2,t})\cdot C_{2,t} &= a_1+a_2(C_2^2-t).
\end{align*}
If $d = 4$, then $C_1^2 = C_2^2 = 0$. This gives
\[
(a_1C_{1,t} + a_2C_{2,t})\cdot (tC_{1,t}  +C_{2,t}) = 0.
\]
By Proposition~\ref{prop: special_divisor_dP_surface_criterion}, $E_t$ cannot be a special divisor over $X$.

Now assume that $d\geq 5$. Then at least one of $C_1^2$ and $C_2^2$ is positive. If $C_1^2 > 0$, pick $a_1 = t$ and $a_2 = 1-\epsilon$ for a sufficiently small $\epsilon >0$. Then
\begin{align*}
(a_1C_{1,t} + a_2C_{2,t})\cdot C_{1,t} &= a_1\left(C_1^2 - \frac{1}{t}\right) + a_2 = tC_1^2 - \epsilon > 0, \\
(a_1C_{1,t} + a_2C_{2,t})\cdot C_{2,t} &= a_1+a_2(C_2^2-t) = a_2C_2^2 + \epsilon > 0.
\end{align*}
By Proposition~\ref{prop: special_divisor_dP_surface_criterion}, $E_t$ is special over $X$. The proof is similar in the case $C_2^2 > 0$.

(c)  Suppose $k\geq 3$. We may assume that $x\in C_1\cap C_2$ and $C_i \cdot C_{i+1} = 1$ for all $1\leq i\leq k$ (where $C_{k+1} = C_1$). Pick $a_1 = a_2 = \epsilon$ for a sufficiently small $\epsilon >0$ and $a_i = 1$ for all $3\leq i\leq k$. Then 
\begin{align*}
\left(\sum_{i=1}^k a_iC_{i,t}\right)\cdot C_{1,t} &= a_1\left(C_1^2 - \frac{1}{t}\right) + a_k > 0, \\
\left(\sum_{i=1}^k a_iC_{i,t}\right)\cdot C_{2,t} &= a_2(C_2^2-t) + a_3 > 0, \\
\left(\sum_{i=1}^k a_iC_{i,t}\right)\cdot C_{j,t} &\geq a_{j-1} + a_{j+1} > 0, \text{ for all }j\geq 3.
\end{align*}
Thus, $E_t$ is special over $X$ by Proposition~\ref{prop: special_divisor_dP_surface_criterion}.
\end{proof}

In the case where some component of $C$ is not nef, we prove the following lemma.

\begin{lemma}\label{lemma:equivalence_of_special_divisors_after_contracting_-1_curves}
Let $X$ be a smooth del Pezzo surface. Let $C\sim -K_X$ be a nodal curve on $X$ such that $(X,C)$ is log canonical. Let $C_0, C_1,\ldots, C_k$ be irreducible components of $C$ with $k\geq 2$. Suppose that $C_0^2 = -1$ and $\pi: X\to X'$ is the birational morphism which contracts $C_1$. Let $E$ be a divisor over $X$. Then the following are equivalent:
\begin{itemize}
\item [(a)] $E$ is an log canonical place of $(X,C)$ and $E$ is special over $X$.
\item [(b)] $E$ is an log canonical place of $(X',C':= \pi(C))$ and $E$ is special over $X'$.
\end{itemize}
\end{lemma}
\begin{proof}
\noindent\textbf{Case 1:} $E$ is contained in $X$. If $E = C_0$, then (a) and (b) are equivalent by Proposition~\ref{prop: special_divisor_dP_surface_criterion}. Thus, it suffices to consider the case $E = C_i$ for some $i\geq 1$.

Suppose that $C_i$ is special over $X$. By Proposition~\ref{prop: special_divisor_dP_surface_criterion}, there exists an effective $\QQ$-divisor $B$ such that $\supp(B) = \supp(C - C_i)$ and $B\cdot C_j > 0$ for any $0\leq j\leq k$ and $j\neq i$. Let $B' = \pi_*B$. Then $B'$ is big and nef, $\supp(B') = \supp(C' - C_i')$ and $B' \cdot C_j' \geq B\cdot C_j > 0$ for $1\leq j\leq k$ and $j\neq i$. Thus, $C_i'$ is special over $X'$ by Proposition~\ref{prop: special_divisor_dP_surface_criterion}.

Suppose that $C_i'$ is special over $X'$. By Proposition~\ref{prop: special_divisor_dP_surface_criterion}, there exists an effective $\QQ$-divisor $B'$ on $X'$ such that $\supp(B') = \supp(C' - C_i')$ and $B'\cdot C_j' > 0$ for any $0\leq j\leq k$ and $j\neq i$. Let $B = \pi^*B' -\epsilon C_0$ for some sufficiently small $\epsilon > 0$. Then $B$ is effective and $\supp(B) = \supp(C-C_i)$. Furthermore, we have
\[
B\cdot C_0 = -\epsilon C_0^2 > 0
\]
and for $1\leq j\leq k$, $j\neq i$,
\[
B\cdot C_j = B'\cdot C_j' - \epsilon C_0\cdot C_j > 0.\]
By Proposition~\ref{prop: special_divisor_dP_surface_criterion}, $C_i$ is special over $X$.

\noindent\textbf{Case 2:} $E$ is not contained in $X$.

We first prove that (a) implies (b). Let $Y\to X$ be the birational morphism extracting $E$. For $1\leq i\leq k$, let $C_{i,Y}$ be the strict transform of $C_i$ on $Y$. By Proposition~\ref{prop: special_divisor_dP_surface_criterion}, there exists $a_1,\ldots,a_k>0$ such that  $(\sum_i a_iC_{i,Y})\cdot C_{i,Y}>0$ for all $1\leq i\leq k$. Let $f:Y\to Y'$ be the birational morphism which contracts $C_{1,Y}$. Let $C_{i,Y'}$ be the image of $C_{i,Y}$ for $2\leq i\leq k$. Then 
\[
f^*\left(\sum_{i=2}^k a_iC_{i,Y'}\right) = \sum_{i=1}^k a_iC_{i,Y} + a_1' C_{1,Y}\]
for some $a_1' > 0$. Thus, for $2\leq i\leq k$,
\[
\left(\sum_{i=2}^k a_iC_{i,Y'}\right)\cdot C_{i,Y'} \geq \left(\sum_{i=2}^k a_iC_{i,Y}\right)\cdot C_{i,Y} >0.
\]
Since $\sum_{i=2}^k C_{i,Y'}$ is the strict transform of $\pi(C)$ on $Y'$, (b) follows from Proposition~\ref{prop: special_divisor_dP_surface_criterion}.

Next, we prove that (b) implies (a). Take $Y'$ and $Y$ as in the previous paragraph. By Proposition~\ref{prop: special_divisor_dP_surface_criterion}, there exists $a_2,\ldots,a_k>0$ such that 
\[
\left(\sum_{i=2}^k a_iC_{i,Y'}\right)\cdot C_{i,Y'}  >0
\]
for all $2\leq i\leq k$. Let $a_1 > 0$ such that
\[
\sum_{i=1}^k a_i C_{i,Y} = f^*\left(\sum_{i=2}^k a_iC_{i,Y'}\right).\]
Then, for any $2\leq i\leq k$ and sufficiently small $\epsilon > 0$,
\[
\left((a_1-\epsilon) C_{1,Y} + \sum_{i=2}^k a_i C_{i,Y} \right)\cdot C_{i,Y} = \left(\sum_{i=2}^k a_iC_{i,Y'}\right)\cdot C_{i,Y'} - \epsilon C_{1,Y} \cdot C_{i,Y} > 0.
\]
We also have
\[
\left((a_1-\epsilon) C_{1,Y} + \sum_{i=2}^k a_i C_{i,Y} \right)\cdot C_{1,Y} = -\epsilon C_{1,Y}^2 > 0.
\]
This proves (a) by Propostion~\ref{prop: special_divisor_dP_surface_criterion}.
\end{proof}

\subsection{Special valuations over $X$, Part I}
The goal of this subsection is to classify valuations in $\text{QM}(X,C)$ which are special over $X$, under the additional assumption that every component of $C$ is nef.
Throughout this subsection, we consider the following set-up.\\

\noindent\textbf{Set-up.}
Let $X$ be a smooth del Pezzo surface and $C\sim -K_X$ be a normal crossing divisor on $X$. Let $x\in C$ be a node. For $t>0$, let $v_t$ denote the quasi-monomial valuation centered at $x$ with weight $(1,t)$ along the two branches of $C$. For $t\in \QQ$, let $E_t$ denote the divisor over $X$ corresponding to the divisorial valuation $v_t$, and let $Y_t\to X$ be the birational morphism extracting the divisor $E_t$. Suppose $C = C_1 + \cdots + C_k$ where $C_i$ is irreducible and nef for all $1\leq i\leq k$. Let $C_{i,t}$ be the strict transform of $C_i$ on $Y_t$.

\begin{lemma} \label{lemma: positivity_of_intersection_product}
Suppose $I$ is an open interval such that $E_t$ is a special divisor for all $t\in I\cap \QQ$. Then for any real number $t_0\in I$, there exists an open interval $I'\subseteq I$ containing $t_0$ and $a_1,\ldots, a_k>0$ such that for all $t\in I'\cap \QQ$ and $1\leq j\leq k$, we have
\[
\left(\sum_{i=1}^k a_i C_{i,t}\right)\cdot C_{j,t} > 0.
\]
\end{lemma}
\begin{proof}
\noindent\textbf{Case 1:} $k = 1$.

In this case, $I \subseteq I_d =  \left(\frac{d-2-\sqrt{d^2-4d}}{2}, \frac{d-2+\sqrt{d^2-4d}}{2}\right)$ by Corollary~\ref{cor:special_divisors_all_boundaries}. Pick $I' = I$. Then for any $t\in I'\cap \QQ$, we have
\[
C_t^2 = d - \frac{(1+t)^2}{t} > 0.\]
\\
\noindent\textbf{Case 2:} $k = 2$. 

By Corollary~\ref{cor:special_divisors_all_boundaries}, $d\geq 5$, so $C_1^2 + C_2^2 = d-4 > 0$. We may assume that $C_1^2 > 0$. Fix $t_0\in I$ and pick $a_1 = t_0$, $a_2 = 1 - \epsilon$, where $\epsilon = \frac{1}{2}\min\{1,t_0\}$. Then for any $t\in \QQ$, we have
\begin{align*}
(a_1C_{1,t} + a_2C_{2,t})\cdot C_{1,t} &= a_1\left(C_1^2 - \frac{1}{t}\right) + a_2 \\
(a_1C_{1,t} + a_2C_{2,t})\cdot C_{2,t} &= a_1+a_2(C_2^2-t).
\end{align*}
Consider the function $f_1, f_2: \RR_{>0}\to \RR$ such that $f_1(t) = a_1(C_1^2 - \frac{1}{t}) + a_2$ and $f_2(t) = a_1+a_2(C_2^2-t)$. Then
\begin{align*}
f_1(t_0) &= t_0C_1^2 - \epsilon  \geq t_0C_1^2 = \frac{1}{2}t_0 > 0,\\
f_2(t_0) &= (1-\epsilon)C_2^2 + \epsilon t_0 > 0.
\end{align*}
Since $f_1$ and $f_2$ are continuous functions, there exists an open interval $I'\subseteq I$ containing $t_0$ such that $f_1(t), f_2(t) > 0$ for all $t\in I'$. \\

\noindent\textbf{Case 3:} $k \geq 3$. 

We may assume that $x\in C_1\cap C_2$ and $(C_i\cdot C_{i+1}) = 1$ for $1\leq i\leq k$ (where $C_{k+1} = C_1$). Fix $t_0>0$ and pick $a_1 = a_2 = \epsilon := \frac{1}{2}\min\{t_0, \frac{1}{t_0}\}$ and $a_j = 1$ for all $j\geq 3$. Similar to the previous case, for $1\leq j\leq k$, let $f_j:\RR_{>0} \to \RR$ be the continuous function such that
\[
f_j(t) =\left(\sum_{i=1}^k a_i C_{i,t}\right) \cdot C_{j,t}
\]
for all $t\in \QQ_{>0}$. Then
\begin{align*}
f_1(t_0) &= \epsilon\left(C_1^2 - \frac{1}{t_0}\right) + 1 > 0,\\
f_2(t_0) &= \epsilon(C_2^2 - t_0) + 1 > 0, \\
f_j(t_0) &\geq a_{j-1} + a_{j+1} > 0, \text{ for all }j\geq 3.
\end{align*}
Thus, by continuity of $f_j$'s, there exists an open interval $I'\subseteq I$ containing $t_0$ such that $f_j(t) > 0$ for all $t\in I'$ and $1\leq j\leq k$.
\end{proof}
\begin{remark}\label{remark:intersection_product_is_strictly_positive_continuous_function}
In fact, the same proof shows that for every $1\leq j\leq k$,
\[
\left(\sum_{i=1}^k a_iC_{i,t}\right)\cdot C_{j,t},
\]
viewed as a function of $t$, extends to a strictly positive continuous function on $I'$.
\end{remark}

\begin{lemma}\label{lemma:big_and_nef_and_not_ample_equivalent_to_existence_of_unicuspidal_curves}
Let $t=\frac{q}{p}\in \QQ_{>0}$. Suppose there exists $a_1,\ldots, a_k>0$ such that 
\[\left(\sum_{i=1}^k a_i C_{i,t}\right)\cdot C_{j,t} > 0\] for all $1\leq j\leq k$ and $\sum_i a_i C_{i,t}$ is not ample. Then, there exists an irreducible curve $D$ on $X$ such that
\begin{itemize}
\item $-K_X \cdot D = p + q$,
\item $D^2 = pq -1$, and
\item $D$ is a $(p,q)$-unicuspidal rational curve on $X$ which is well-formed with respect to $(C,x)$.
\end{itemize}
\end{lemma}
\begin{proof}
Let $B_t = \sum_{i=1}^k a_iC_{i,t}$. Then $B_t$ is big and nef. If $B_t$ is not ample, then there exists an irreducible curve $D_t$ on $Y_t$ such that $B_t\cdot D_t = 0$. Note that $D_t\neq C_{j,t}$ for any $1\leq j\leq k$. Thus, $D_t\cdot C_{j,t} = 0$ for all $j$. Let $E_t$ be the exceptional divisor of $f_t$ and $D$ be the image of $D_t$ on $X$. By adjunction formula,
\begin{align*}
-2\leq (K_{Y_t} + D_t) \cdot D_t = \left(-\sum_{i}C_{i,t} - E_t + D_t\right) \cdot D_t = D_t^2 - E_t\cdot D_t < -E_t\cdot D_t,
\end{align*}
so $0\leq D_t\cdot E_t< 2$. Note that $D_t\cdot E_t > 0$ because $f_t^*(\sum_{i}C_i)\cdot D_t = -K_X \cdot D> 0$. Furthermore, $D_t$ and $E_t$ intersect at smooth points of $Y_t$ because every singular point of $Y_t$ lies on the support of $\sum_i C_{i,t}$. Thus, $D_t\cdot E_t = 1$. This implies that
\[
-K_X\cdot D =  f_t^*\left(\sum_i C_i\right)\cdot D_t =\left( (p+q)E_t + \sum_i C_{i,t}\right) \cdot D_t = p+q
\]
and
\[
f_t^*D = D_t + pq E_t\]
since $E_t^2 = -\frac{1}{pq}$. Thus, $D_t^2 = D^2 - pq$ is an integer. Since $-2\leq D_t^2 - E_t\cdot D_t$ (by adjunction) and $D_t^2 < 0$, we have $D_t^2 = -1$ and $D^2 = pq-1$. The equality 
\[
(K_{Y_t} + D_t) \cdot D_t = -2
\]
also implies that $D_t$ is a smooth rational curve on $Y_t$. In particular, $D$ is smooth or has a unique singular point at $x$.

Next, we show that $D$ is a $(p,q)$-unicuspidal curve well-formed with respect to $(C,x)$. Let $x_1, x_2$ be a local (analytic) coordinate near $x$ such that the local equation of $C$ is $x_1x_2 = 0$. Since $D$ is irreducible, we can write the local equation of $D$ as
\[
c_1x_1^\alpha + c_2 x_2^\beta + g(x_1,x_2)
\]
for some $c_1,c_2\neq 0$ and a power series $g$, such that the coefficients of $x_1^i$ and $x_2^j$ in $g$ vanish whenever $i\leq \alpha$ and $j\leq \beta$. Then the local intersection product at $x$ satisfies
\[
(D\cdot C)_x = \alpha + \beta.
\]
Let $v = \ord_{E_t}\in \text{QM}(X,C)$. Then 
\[
v(D) = v(f_t^*D) = pq.\]
Since $v(D)\leq p\alpha, q\beta$ by the definition of a quasi-monomial valuation, we have $\alpha \geq p$ and $\beta \geq q$. However,
\[
p+q = D\cdot C \geq (D\cdot C)_x = \alpha + \beta \geq p  + q,
\]
so $\alpha = q$ and  $\beta = p$. This implies that the Newton polygon of $D$ in $x_1, x_2$ coordinate is precisely the region above the line connecting $(q,0)$ and $(0,p)$, so $D$ is a $(p,q)$-unicuspidal curve which is well-formed with respect to $(C,x)$.
\end{proof}

We need the following property about the set of unicuspidal curves on $X$ well-formed with respect to $(C,x)$.

\begin{lemma}\label{lemma:accumulation_points_unicuspidal_curves}
Let $T$ denote the set of $t= \frac{q}{p}$ such that
\begin{itemize}
\item There exists a $(p,q)$-unicuspidal curve $D$ on $X$ which is well-formed with respect to $(C,x)$, 
\item $-K_X\cdot D = p+q$,
\item $D^2 = pq-1$, and
\item $E_{t}$ is a special divisor over $X$.
\end{itemize}
\item [(a)] Suppose $k=1$. Then $T$ does not have any accumulation points except possibly at $\frac{d-2\pm\sqrt{d^2-4d}}{2}$.
\item [(b)] Suppose $k\geq 2$. Then $T$ does not have any accumulation points except possibly at $0$ and $\infty$. 
\item [(c)] Suppose $k = 2$ and $C_1^2, C_2^2 > 0$. Then $T$ is finite.
\item [(d)] Suppose $k\geq 3$. Then $T = \varnothing$. 
\end{lemma}
\begin{proof}
(a) Since $E_t$ is special, by Proposition~\ref{cor:special_divisors_all_boundaries}, $d\geq 5$ and $t\in I_d$. This implies that $(p+q)^2 < dpq$. By the Hodge index theorem, 
\[
d(pq-1) = (C^2) (D^2) \leq (C\cdot D)^2 = (p+q)^2 < dpq.
\]
Thus, $(p+q)^2 = dpq - m$ for some integer $1\leq m\leq d$. For each fixed $m$, the set of $\frac{q}{p}$ for which $(p,q)$ is a solution of $(p+q)^2 = dpq-m$ does not have any accumulation points except possibly at $\frac{d-2\pm \sqrt{d^2-4d}}{2}$. \\

(b) We may assume that $p,q\geq 2$. Suppose $x\in C_1\cap C_2$. Since $D$ is well-formed with respect to $C_1 + C_2$, the local intersection products satisfy $(D\cdot C_1)_x \geq p$ and $(D\cdot C_2)_x \geq q.$ Since $D$ is nef and $-K_X\cdot D = p+q$, we must have $D\cdot C_1 = p$, $D\cdot C_2 = q$, $D\cdot C_j = 0$ for all $j\geq 3$. \\

\noindent\textbf{Case 1:} $k = 2$.

Since $a_1C_1 + a_2C_2$ is nef for all $a_1,a_2\geq 0$, by the Hodge index theorem,
\[
(a_1C_1 + a_2C_2)^2 (D^2)\leq ((a_1C_1 + a_2C_2) \cdot D)^2.\]
This implies that for all $a_1,a_2\geq 0$,
\[
a_1^2(p^2 - (pq-1)C_1^2) + a_2^2(q^2 - (pq-1)C_2^2) + 2a_1a_2 (2-pq) \geq 0.\]
Thus, $p^2 \geq (pq-1)C_1^2$, $q^2 \geq (pq-1)C_2^2$, and $(p^2 - (pq-1)C_1^2)(q^2 - (pq-1)C_2^2)\geq (pq-2)^2$. There are several cases:
\begin{itemize}
\item Suppose $C_1^2 = C_2^2 = 0$. Then $E_t$ is not special for all $t\in \QQ$ by Proposition~\ref{cor:special_divisors_all_boundaries}.
\item Suppose $C_1^2, C_2^2 \geq 1$. Then $p^2\geq pq-1$ and $q^2 \geq pq -1$ imply that $p=q$. This implies that $t =1$.
\item Suppose $C_1^2 = 0$ and $C_2^2>0$. Then
\[
p^2 (q^2 - (pq-1)C_2^2) \geq (pq-2)^2,
\]
which implies that
\[
(pq-1)C_2^2 \leq \frac{4q}{p} - \frac{4}{p^2}.
\]
Thus, $p\leq 2$ and $\infty$ the only possible accumulation point of the set of $\frac{q}{p}$.
\item $C_1^2 > 0$ and $C_2^2 = 0$. Similar to the previous case, $q \leq 2$ and $0$ is the only possible accumulation point of the set of $\frac{q}{p}$.
\end{itemize}

\noindent\textbf{Case 2:} $k\geq 3$.
We may assume that $C_3\cdot C_2 = 1$. Then by the Hodge index theorem, for any $a>0$,
\[
(C_1 + C_2 + aC_3)^2 (D^2) \leq ((C_1 + C_2 + aC_3)\cdot D)^2,
\]
which implies that
\[
2a D^2 \leq ((C_1+C_2)\cdot D)^2.
\]
This is a contradiction.\\

(c) By the computation in (b), we have $p^2\geq pq-1$ and $q^2\geq pq-1$. Thus, $p=q$ or $(p,q) = (1,2), (2,1)$. This implies that $T$ is finite.\\

(d) This follows from the computation in (b).
\end{proof}

\begin{proposition}\label{prop:classification_of_special_valuations_nef_case}
Let $v\in \text{QM}(X,C)$ be a quasi-monomial valuation. Assume $C_i$ is nef for all $1\leq i\leq k$.
\begin{itemize}
\item[(a)] Suppose $C$ is irreducible. Then $v$ is special over $X$ if and only if $v=v_t$ for some $t\in I_d$.
\item[(b)] Suppose $C$ is reducible. Then $v$ is special if and only if the following cases do not happen:
\begin{itemize}
\item $k = 2$ and $d = 4$.
\item $k = 2 $, $d\geq 5$, $C_1^2 = 0$, and $v = \ord_{C_2}$.
\item $k = 2 $, $d\geq 5$, $C_2^2 = 0$, and $v = \ord_{C_1}$.
\end{itemize}
\end{itemize}
\end{proposition}
\begin{proof}
Suppose first that $v \neq \ord_{C_i}$ for any $1\leq i\leq k$. It suffices to consider the case $v = v_{t_0}$ for some $t_0 > 0$. By Lemma~\ref{lemma: positivity_of_intersection_product}, there exists an open interval $I$ containing $t_0$ and $a_1,\ldots, a_k>0$, such that $\sum_i a_iC_{i,t}$ is big and nef for all $t\in I\cap \QQ$. Then, by Lemma~\ref{lemma:accumulation_points_unicuspidal_curves} and Lemma~\ref{lemma:big_and_nef_and_not_ample_equivalent_to_existence_of_unicuspidal_curves}, we may shrink $I$ around $t_0$ such that $\sum_i a_iC_{i,t}$ is ample for all $t\in I\cap \QQ$. By Proposition~\ref{prop: special_valuation_dP_surface_criterion}, $v$ is special.

Next, suppose $v = \ord_{C_i}$ for some $i$. We have the following cases:
\begin{itemize}
\item Suppose $C$ is irreducible. Then $\ord_C$ is not special by Proposition~\ref{prop: special_divisor_dP_surface_criterion}.
\item Suppose $k = 2$. If $C_1^2 = 0$, then $\ord_{C_2}$ is not special by Proposition~\ref{prop: special_divisor_dP_surface_criterion}. Similarly, if $C_2^2 = 0$, then $\ord_{C_1}$ is not special.
\item Suppose $k\geq 3$. Then $\ord_{C_i}$ is special for all $1\leq i\leq k$ by Proposition~\ref{prop: special_divisor_dP_surface_criterion}. 
\end{itemize}
The proof is complete.
\end{proof}

\subsection{Special valuations over $X$, Part II}
The goal of this subsection is to classify valuations in $\text{QM}(X,C)$ which are special over $X$, in the case that some component of $C$ is not nef. This is done by comparing special valuations over $X$ with special valuations over $X'$, where $X'$ is obtained from $X$ after contracting a non-nef component of $C$. The set-up is the following.\\

\noindent\textbf{Set-up.} Let $X$ be a smooth del Pezzo surface. Let $C\sim -K_X$ be a normal crossing divisor on $X$. Let $C_0, C_1,\ldots, C_k$ be irreducible components of $C$ with $k\geq 1$. Suppose that $C_0^2 = -1$ and $\pi: X\to X'$ is the birational morphism which contracts $C_0$. Let $x\in C$ be a node. Denote $C' = \pi(C)$ and $x' = \pi(x).$ For $t>0$, let $v_t$ denote the quasi-monomial valuation centered at $x$ with weight $(1,t)$ along the two branches of $C$. For $t\in \QQ$, let $E_t$ denote the divisor over $X$ corresponding to the divisorial valuation $v_t$, and let $Y_t\to X$ be the birational morphism extracting $E_t$. Let $Y_t\to Y_{t}'$ be the birational morphism contracting the strict transform of $C_{0}$ on $Y_t$. Let $C_{i,t}$ and $C_{i,t}'$ denote the strict transform of $C_i$ on $Y_t$ and $Y_{t}'$, respectively.

\begin{lemma}\label{lemma:equivalence_of_positivity_of_intersection_product_after_contraction}
Suppose that $I$ is an open interval such that $v_t$ is not proportional to the divisorial valuation $\ord_{C_0}$ for any $t\in I$. Let $t_0\in I$. Then the following are equivalent:
\begin{itemize}
\item[(a)] There exists $a_0,\ldots, a_k > 0$, $\delta_1 > 0$, and $I_1\subseteq I$ containing $t_0$, such that for any $t\in I_1\cap \QQ$, $(\sum_{i=0}^k a_i C_{i,t})\cdot C_{j,t} > \delta_1$ for all $0\leq j\leq k$.
\item[(b)] There exists $b_1,\ldots, b_k >0$, $\delta_2 >0$, and $I_2\subseteq I$ containing $t_0$, such that for any $t\in I_2\cap \QQ$, $(\sum_{i=1}^k b_i C_{i,t}')\cdot C_{j,t}' > \delta_2$ for all $1\leq j\leq k$.
\end{itemize}
\end{lemma}
\begin{proof}
We first show that (b) implies (a). Let $\pi_t$ denote the morphism $Y_t\to Y_t'$. For $t\in I_2\cap\QQ$, write
\[
b_0(t)C_{0,t}+\sum_{i=1}^k b_iC_{i,t} =\pi_t^*\left(\sum_{i=1}^k b_iC_{i,t}'\right),
\]
where
\[
b_0(t) = -(C_{0,t}^2)^{-1}\sum_{i=1}^k b_i (C_{i,t}\cdot C_{0,t}).\]
Hence, $b_0$ extends to a continuous function $b_0: I_2\to \RR$. Let $I_1\subseteq I_2$ be a subinterval containing $t_0$ such that $b_0(t)$ is bounded below by an absolute constant $\delta$ when $t\in I_1$. We may assume $\delta < \min \{1, \delta_2\}$.
Then for any $\epsilon \in (\frac{1}{3}\delta, \frac{2}{3}\delta)$ and $t\in I_1$, 
\[
\left((b_0(t) - \epsilon)C_{0,t} + \sum_{i=1}^k b_i C_{i,t}\right) \cdot C_{j,t} = \left(\sum_{i=1}^k b_iC_{i,t}'\right)\cdot C_{j,t}' - \epsilon(t)C_{0,t}\cdot C_{j,t} \geq \delta_2 - \epsilon > \frac{1}{3}\delta
\]
for all $1\leq j\leq k$ and
\[
\left((b_0(t) - \epsilon)C_{0,t} + \sum_{i=1}^k b_i C_{i,t}\right) \cdot C_{0,t} = -\epsilon C_{0,t}^2 > \frac{1}{3}\delta
\]
since $C_{0,t}^2 \leq C_0^2 = -1$.
Because $b_0(t)$ is a continuous function, after replacing $I_1$ by a smaller open neighborhood of $t_0$, there exists a constant $a_0$ such that
\[
b_0(t) - \frac{2}{3}\delta < a_0< b_0(t) - \frac{1}{3}\delta\]
for all $t\in I_1$. Then, statement (a) is satisfied by taking $a_0, a_1 = b_1, \ldots, a_k = b_k$, $\delta_1 = \frac{\delta}{3}$, and $I_1$.

Finally, assume (a) holds. Then (b) holds by taking $b_1 = a_1,\ldots,b_k = a_k$, $\delta_2 = \delta_1$, and $I_2 = I_1$.
\end{proof}

\begin{lemma}\label{lemma:equivalence_of_unicuspidal_curves_after_contraction}
Suppose $D$ is a $(p,q)$-unicuspidal curve on $X$ which is well-formed with respect to $(C,x)$, such that $-K_X\cdot D = p+q$ and $D^2 = pq-1$. Let $D' = \pi(D)$.
\begin{itemize}
\item[(a)] Suppose $x\not\in C_0$. Then $D'$ is a $(p,q)$-unicuspidal curve on $X'$ which is well-formed with respect to $(C',x')$. Furthermore, $-K_{X'}\cdot D' = p+q$ and $D'^2 = pq-1$.
\item[(b)] Suppose $x\in C_0$ and $D\cdot C_0 = p$. Then $D'$ is a $(p,p+q)$-unicuspidal curve on $X'$ which is well-formed with respect to $(C',x')$. Furthermore, $-K_{X'}\cdot D' = 2p+q$ and $D'^2 = p(p+q)-1$.
\item[(c)] Suppose $x\in C_0$ and $D\cdot C_0 = q$. Then $D'$ is a $(q,p+q)$-unicuspidal curve on $X'$ which is well-formed with respect to $(C',x')$. Furthermore, $-K_{X'}\cdot D' = p+2q$ and $D'^2 = q(p+q)-1$.
\end{itemize}
\end{lemma}
\begin{proof}
(a) Since the local intersection product satisfies $(\sum_{i=1}^k C_i \cdot D)_x = p+q$, $D\cap C_0 = \varnothing$. As a result, $\pi$ is an isomorphism in a neighborhood of $D$ and $D'$. This implies all the statements in (a).\\

(b) Note that $\pi^*D' = D + pC_0$. Then
\[
-K_{X'}\cdot D' = -K_X \cdot \pi^*D' = -K_X\cdot (D+pC_0) = 2p+q,
\]
and
\[
D'^2 =  (D+pC_0)^2 = p(p+q) -1.
\]
The fact that $D'$ is a $(p,p+q)$-unicuspidal curve on $X'$ well-formed with respect to $(C',x')$ follows from the explicit description of a blow-up.\\

(c) The proof is exactly the same as (b).
\end{proof}

\begin{lemma}\label{lemma:equivalence_of_special_valuations_after_contraction}
Let $t_0 \not\in \QQ$. If $v_{t_0}$ is a special valuation over $X'$, then
\begin{itemize}
\item [(a)] There exists $b_1,\ldots, b_k >0$, $\delta >0$, and an open interval $I$ containing $t_0$, such that for any $t\in I\cap \QQ$, $(\sum_{i=1}^k b_i C_{i,t}')\cdot C_{j,t}' > \delta$ for all $1\leq j\leq k$.
\item [(b)] $v_{t_0}$ is a special valuation over $X$.
\end{itemize}
\end{lemma}

\begin{proof}
We first prove this proposition assuming that every component of $C'$ is nef. Then (a) follows from the proof of Proposition~\ref{prop:classification_of_special_valuations_nef_case} and Remark~\ref{remark:intersection_product_is_strictly_positive_continuous_function}. For (b), by Lemma~\ref{lemma:equivalence_of_positivity_of_intersection_product_after_contraction}, Lemma~\ref{lemma:big_and_nef_and_not_ample_equivalent_to_existence_of_unicuspidal_curves}, and Proposition~\ref{prop: special_valuation_dP_surface_criterion}, it suffices to show that there does not exists an infinite sequence $\{\frac{q_n}{p_n}:n\geq 1\}$ approaching $t_0$, such that for all $n\geq 1$, there exists a $(p_n,q_n)$-unicuspidal rational curve on $X$ which is well-formed with respect to $(C,x)$.

Assume that such a sequence $\{\frac{q_n}{p_n}:n\geq 1\}$ exists. By Lemma~\ref{lemma:equivalence_of_unicuspidal_curves_after_contraction}, after passing to a subsequence, we may assume that one of the following holds:
\begin{itemize}
\item[(i)] For every $n\geq 1$, there exists a $(p_n, q_n)$-unicuspidal curve on $X'$ well-formed with respect to $(C',x')$.
\item[(ii)] For every $n\geq 1$, there exists a $(p_n, p_n+q_n)$-unicuspidal curve on $X'$ well-formed with respect to $(C',x')$.
\item[(iii)] For every $n\geq 1$, there exists a $(q_n,p_n+ q_n)$-unicuspidal curve on $X'$ well-formed with respect to $(C',x')$.
\end{itemize}
Without loss of generality, we assume case (ii) holds (the other two cases are similar). Then the valuation $v_{t_0}$ is a quasi-monomial valuation on $X'$ with weight $(1,1+t_0)$ centered at $x'$ along the two branches of $C'$. However, the fact that $\lim_{n\to\infty}\frac{p_n+q_n}{p_n} = 1+t_0$ contradicts with Lemma~\ref{lemma:accumulation_points_unicuspidal_curves} and Proposition~\ref{prop:classification_of_special_valuations_nef_case} combined. This finishes the proof of (b).

Finally, the general case follows from an induction on the number of components $C_i'$ which is not nef.
\end{proof}

\begin{corollary}\label{cor:equivalence_of_special_valuations_after_contraction}
Let $X$ be a smooth del Pezzo surface of degree $d$. Let $C\sim -K_X$ be a normal crossing divisor on $X$. Suppose $X\to X''$ is the birational morphism which satisfies the following conditions:
\begin{itemize}
\item $X\to X''$ only contracts irreducible components of $C$.
\item Every component of $C''$ is nef, where $C''$ is the image of $C$ on $X''$.
\end{itemize}
Let $v\in \text{QM}(X,C) \cong \text{QM}(X'',C'')$. Then $v$ is a special valuation over $X$ if and only if $v$ is a special valuation over $X''$.
\end{corollary}
\begin{proof}
If $v$ is a divisorial valuation, then we are done by Lemma~\ref{lemma:equivalence_of_special_divisors_after_contracting_-1_curves}. Thus, we may assume that $v$ is not divisorial.

By Lemma~\ref{lemma:equivalence_of_special_valuations_after_contraction}, if $v$ is a special over $X''$, then $v$ is special over $X$. If $v$ is special over $X$, then by Theorem~\ref{thm:rational_space_containing_special_valuation_is_special}, there exists an open neighborhood $I$ of $v$ in $\text{QM}(X,C)$ in which every divisorial valuation is special over $X$. By Lemma~\ref{lemma:equivalence_of_special_divisors_after_contracting_-1_curves}, every divisorial valuation in $I$ is also special over $X'$. Finally, $v$ is special by the classification in Proposition~\ref{prop:classification_of_special_valuations_nef_case}.
\end{proof}

\subsection{Proof of Theorem~\ref{mainthm:classification_of_special_valuations_on_del_pezzo_surfaces} and \ref{mainthm:structure_of_the_space_of_special_valuations}}
\begin{proof}[Proof of Theorem~\ref{mainthm:classification_of_special_valuations_on_del_pezzo_surfaces}]
Part (a),(b), and (c) follows from Propostition~\ref{prop:classification_of_special_valuations_nef_case}. Part (d) follows from Corollary~\ref{cor:equivalence_of_special_valuations_after_contraction}.
\end{proof}

\begin{proof}[Proof of Theorem~\ref{mainthm:structure_of_the_space_of_special_valuations}]
Part (a) follows from Theorem~\ref{mainthm:classification_of_special_valuations_on_del_pezzo_surfaces}. 

For (b), let $\phi: X\to X'$ be the morphism which only contracts components of $C$ and every component of $C' = \phi(C)$ is nef. Then $\mathcal{D}^{KV}(X,C) = \mathcal{D}(X,C)$ implies that $\mathcal{D}^{KV}(X',C') = \mathcal{D}(X',C')$. Then, by Proposition~\ref{prop:classification_of_special_valuations_nef_case}, $C'$ has at least 3 irreducible components, or $C'$ has two irreducible components $C_1,C_2$ such that $C_1^2, C_2^2 > 0$. Then, by Lemma~\ref{lemma:accumulation_points_unicuspidal_curves}, the set $T$ of $\frac{q}{p}$, such that there exists a $(p,q)$-unicuspidal rational curve on $X$ well-formed with respect to $(C, x)$ for some node $x$ of $C$, is finite.
We partition $\mathcal{D}^{KV}(X',C')\sim S^1$ by finitely many open intervals, such that each boundary point of these open intervals is 
\begin{itemize}
\item either $\ord_{C_i'}$, where $C_i'$ is an irreducible component of $C'$,
\item or a divisorial valuation which corresponds to a $(p,q)$ weighted blow-up of $X'$ along two branches of $C'$ at a node $x'$ of $C'$, such that the weights $p$ and $q$ satisfy $\frac{q}{p}\in T$.
\end{itemize}
Then, the statement follows from Lemma~\ref{prop: special_valuation_dP_surface_criterion}, Lemma~\ref{lemma: positivity_of_intersection_product}, and Lemma~\ref{lemma:big_and_nef_and_not_ample_equivalent_to_existence_of_unicuspidal_curves}.

For (c), the proof is similar to (b), except that $T$ may not be finite. By Lemma~\ref{lemma:accumulation_points_unicuspidal_curves}, the accumulation points of $T$ are contained in the set of boundaries of $\mathcal{D}^{KV}(X',C')$ 
roposition~\ref{prop: special_valuation_dP_surface_criterion}. This shows that the set of boundaries of these partitioning intervals $I_k$ do not have accumulation points except possibly at the boundaries of $\mathcal{D}^{KV}(X',C')$.

Part (d) follows from Lemma~\ref{lemma:big_and_nef_and_not_ample_equivalent_to_existence_of_unicuspidal_curves}.
\end{proof}

\section{Qdlt Fano Type Models of Special Degenerations}
\subsection{Notations and outline of the proof}
Throughout this section, we will consider Construction~\ref{construction: special_degeneration} in the following set-up.\\

\noindent\textbf{Set-up.} Let $X$ be a klt Fano surface and $C\sim_{\QQ} -K_X$ be a $\QQ$-divisor such that $(X,C)$ is log canonical. Assume $(X,C)$ is a normal crossing pair in a neighborhood of $x\in X$. Assume $E$ is a divisor on $X$ such that
\begin{itemize}
\item $E$ is a log canonical place of $(X,C)$,
\item The center of $E$ on $X$ is $x$,
\item $E$ is special over $X$.
\end{itemize}
Under these assumptions, $E$ is the exceptional divisor of a weighted blow-up of $X$ at $x$ with weights $p,q$ along the two branches of $C$, for some coprime positive integers $p$ and $q$.\\

Let 
\[
X\times \AA^1 \leftarrow \mathcal{Y}  \dashrightarrow \mathcal{X}'\dashrightarrow \mathcal{X}^m \to \mathcal{X}
\]
be the diagram in Construction~\ref{construction: special_degeneration} for the pair $(X,C)$ and its log canonical place $E$. When restricting to the central fibers, we obtain the following diagram
\[
X \leftarrow X_1\cup S_1 \dashrightarrow  S_2\dashrightarrow S_3\to X',
\]
where $S_2, S_3, X'$ are the central fibers of $\mathcal{X}', \mathcal{X}^m, \mathcal{X}$, respectively.
\begin{lemma}\label{lemma:Y_is_1,p,q_weighted_blow_up} 
We may take $\mathcal{Y}\to X\times \AA^1$ to be the $(1,p,q)$-weighted blow-up at $x\times \{0\}$ along the normal crossing divisor $X\times \{0\} + C$ (with weight 1 along $X\times \{0\}$ and weights $p,q$ along the two branches of $C$). Then $S_1\cong \PP(1,p,q)$. Let $u_0,u_1,u_2$ be the homogeneous coordinates on $S_1$ of degree $1,p,q$, respectively. Then the strict transform of $C\times \AA^1$ on $\mathcal{Y}$, when restricted to $S_1$, is the divisor $(u_1u_2) = 0$.
\end{lemma}
\begin{proof}
This follows from Construction~\ref{construction: special_degeneration} and Lemma~\ref{lemma:weighted_blow_up_nc_divisor}.
\end{proof}

\begin{definition}\label{def:notations_of_construction_of_special_degeneration}
{\em
For the rest of this section, we will follow the notations defined below. Notations of divisors are guided by the following principles: (1) mathcal letters (including the decorated ones) are only used to denote a family of surfaces over $\AA^1$, (2) for a divisor $D$ on $Z$ and a birational model $Z'$ of $Z$, we usually use $D_{Z'}$ to denote the strict transform of $D$ on $Z'$.
\begin{itemize}
\item As in Lemma~\ref{lemma:Y_is_1,p,q_weighted_blow_up}, we take $\mathcal{Y}\to X\times \AA^1$ is the $(1,p,q)$-weighted blow-up at $x\times \{0\}$ along the normal crossing divisor $X\times \{0\} + C$. We denote the homogeneous coordinates of $S_1\cong \PP(1,p,q)$ as $u_0,u_1,u_2$. 
\item For each $i=0,1,2$, define $L_i = (u_i = 0) \subseteq S_1$.
\item For each birational model $S$ of $S_1$ and $i=0,1,2$, define $L_{i, S}$ to be the strict transform of $L_i$ on $S$.
\item For each birational model $\mathcal{Z}$ of $X\times \AA^1$, define $\mathcal{C}_{\mathcal{Z}}$ to be the strict transform of $C\times \AA^1$ on $\mathcal{Z}$. If $Z$ is the central fiber of $\mathcal{Z}$, define $C_{Z}$ to be the restriction of $\mathcal{C}_{\mathcal{Z}}$ to $Z$. 
\item Define $x_1\in \mathcal{Y}$ to be the intersection of the closure of $\{x\}\times (\AA^1-\{0\})$ in $\mathcal{Y}$ and the central fiber of $\mathcal{Y}$. We similarly define $x_2\in \mathcal{X}'$ and $x_3\in \mathcal{X}^m$.
\end{itemize}
}
\end{definition}

The main goal of this section is to construct the following commutative diagram
\begin{equation}\label{equation:construction_of_special_degeneration}
\begin{tikzcd}
    \Tilde{X} \times \AA^1 \arrow{d}& \arrow{l}\Tilde{\mathcal{Y}}   \arrow{d}{} \arrow[dashed]{r}&  \Tilde{\mathcal{X}'}\arrow{d}{} \arrow[dashed]{r}& \Tilde{\mathcal{X}}^m \arrow{d}\\
     X\times \AA^1 & \arrow{l} \mathcal{Y}   \arrow[dashed]{r}  &  \mathcal{X}' \arrow[dashed]{r}&  \mathcal{X}^m \arrow{r}& \mathcal{X}
\end{tikzcd}
\end{equation}
where
\begin{itemize}
\item The bottom row is Construction~\ref{construction: special_degeneration} for the pair $(X,C)$ and the log canonical place $E$.
\item Each vertical arrow is a birational morphism with a unique exceptional divisor. This exceptional divisor is either $E\times \AA^1$ (for the first vertical arrow) or the strict transform of $E\times \AA^1$ (for the other vertical arrows).
\item The top row can be constructed by MMPs which are similar to Construction~\ref{construction: special_degeneration}.
\end{itemize}
We denote the central fibers of \ref{equation:construction_of_special_degeneration} by the following diagram
\begin{equation}\label{equation: construction_of_special_degeneration_central_fiber}
\begin{tikzcd}
    \Tilde{X} \arrow{d}& \arrow{l} \Tilde{X}_1 \cup \Tilde{S}_1   \arrow{d}{} \arrow[dashed]{r}& \Tilde{S}_2\arrow{d}{} \arrow[dashed]{r}& \Tilde{S}_3 \arrow{d}\\
     X & \arrow{l} X_1 \cup S_1   \arrow[dashed]{r}  &  S_2 \arrow[dashed]{r}&  S_3 \arrow{r}& X'
\end{tikzcd}
\end{equation}

We now state the main theorem of this section.
\begin{theorem}\label{thm:qdlt_Fano_type_model_on_special_degeneration}
There exists commutative diagrams ~(\ref{equation:construction_of_special_degeneration}) and (\ref{equation: construction_of_special_degeneration_central_fiber}) such that the following statement holds:
\begin{itemize}
\item [(a)]  The bottom rows of (\ref{equation:construction_of_special_degeneration}) and (\ref{equation: construction_of_special_degeneration_central_fiber}) are the same as Construction~\ref{construction: special_degeneration} for the pair $(X,C)$ and the log canonical place $E$.
\item [(b)] Each vertical arrow of (\ref{equation:construction_of_special_degeneration}) is a birational morphism with a unique exceptional divisor. This exceptional divisor is either $E\times \AA^1$ (for the first vertical arrow) or the strict transform of $E\times \AA^1$ (for the other vertical arrows).
\item [(c)] Each vertical arrow of (\ref{equation: construction_of_special_degeneration_central_fiber}) is a birational morphism with a unique exceptional divisor. This exceptional divisor is either $E$ (for the first vertical arrow) or the strict transform of $E$ (for the other vertical arrows).
\item [(d)] In (\ref{equation:construction_of_special_degeneration}), the birational maps $\mathcal{Y}\dashrightarrow \mathcal{X}'$ and $\mathcal{X}'\dashrightarrow \mathcal{X}^m$ are isomorphisms in a neighborhood of the closure of $\{x\}\times (\AA^1\setminus \{0\})$.
\item [(e)] $(S_3, C_{S_3})$ is a log canonical log Calabi-Yau pair which is simple normal crossing in a neighborhood of $x_3\in S_3$.
\item [(f)] For $i=1,2$, $(\Tilde{S}_3, L_{i, \Tilde{S}_3} + E_{\Tilde{S}_3}) \to X'$ is a qdlt Fano type model.
\end{itemize}
\end{theorem}

\subsection{Construction of diagram (\ref{equation:construction_of_special_degeneration})}
The goal of this section is to construct the commutative diagram $(\ref{equation:construction_of_special_degeneration})$ and prove parts (a)-(e) of Theorem~\ref{thm:qdlt_Fano_type_model_on_special_degeneration}.

\begin{lemma}\label{lemma:factorize_MMP_into_flips_and_divisorial_contraction}
The rational map $S_1\dashrightarrow S_2$ can be factorized as $S_1\leftarrow S_{12} \rightarrow S_2$, such that the following properties hold.
\begin{itemize}
\item[(a)] For every exceptional divisor $F$ of $S_{12}\to S_1$, the center of $F$ on $S_1$ is contained in $L_0 = (u_0=0)\subseteq S_1$. Furthermore, if the center of $F$ on $S_1\cong \PP(1,p,q)$ is $[0,1,0]$ or $[0,0,1]$, then $F$ is torus-invariant (with respect to the standard torus action on $\PP(1,p,q)$).
\item[(b)] $S_{12} \to S_2$ is the divisorial contraction of the strict transform of $L_0$.
\end{itemize}
\end{lemma}
\begin{proof}
Suppose the MMP in Construction~\ref{construction: special_degeneration}.(2) has the form
\[
\mathcal{Y} = \mathcal{Y}_0 \dashrightarrow \mathcal{Y}_1\dashrightarrow \cdots \dashrightarrow \mathcal{Y}_n = \mathcal{X}'.
\] Since this is a $-S_1$-MMP, the last step $\mathcal{Y}_{n-1}\to \mathcal{Y}_n$ is a divisorial contraction which contracts the strict transform of $X_1$. For $i< n$, $\mathcal{Y}_{i-1}\dashrightarrow \mathcal{Y}_{i}$ is a flip which contracts a curve in $X_1$ and extracts a curve over $S_1$. Let $S_{12}$ be the component of the central fiber of $\mathcal{Y}_{n-1}$ which dominates $S_2$. Then $S_{12}\to S_2$ contracts the strict transform of $L_0$ and $S_{12}\to S_1$ is a birational morphism.

Let $F$ be an exceptional divisor of $S_{12}\to S_1$. Then $F$ is a $\GG_m$-equivariant divisor over $S_1\cong \PP(1,p,q)$, where $\GG_m$ acts on the coordinates $u_0, u_1, u_2$ of $\PP(1,p,q)$ with weights $-1,0,0$. The center of $F$ is contained in $L_0$ because $F$ is a $-S_1\sim_{\QQ,\AA^1} X_1$-positive curve and $X_1\cap S_1 = L_0$. The last part of (a) follows from Lemma~\ref{lemma:G_m_invariant_divisors_on_P(a,b,c)} below.
\end{proof}

\begin{lemma}\label{lemma:G_m_invariant_divisors_on_P(a,b,c)}
Let $X = \PP(a,b,c)$ be a weighted projective space with coordinates $u_0, u_1, u_2$. Suppose $\GG_m$ acts on $X$ by acting on the coordinates $u_0, u_1, u_2$ with weights $-1, 0, 0$, respectively. Suppose that $F$ is a $\GG_m$-invariant exceptional divisor centered at the point $[0,1,0]$ or $[0,0,1]$. Then $F$ is invariant with respect to the standard torus action on $X$.
\end{lemma}

\begin{proof}
Let $v = \ord_F$ be a $\GG_m$-invariant valuation on $X$. Suppose the center of $v$ is $[0,1,0]$ (the case $[0,0,1]$ can be proved similarly). Then $v(u_0), v(u_2) > 0$ and $v(u_1) = 0$. The coordinate ring of the affine open set $U := (u_1\neq 0)$ is given by 
\[S = \CC\left[\left\{\frac{u_0^\alpha u_2^\gamma}{u_1^\beta}: \alpha,\beta,\gamma \geq 0, a\alpha+c\gamma = b\beta\right\}\right]\]
Let $f \in S$. Then we may write $f$ as
\[
f = \sum_{\alpha} u_0^\alpha f_\alpha(u_1^{-1},u_2),
\]
where $f_\alpha$ is a polynomial in $u_1^{-1}, u_2$ with homogeneous degree $-a\alpha$. Since $v$ is $\GG_m$-invariant, for any $t\in \GG_m$,
\[
v(f) = v(t\cdot f) = v\left(\sum_\alpha t^{-\alpha}u_0^{\alpha} f_\alpha\right).
\]
By taking various $t$'s and performing Gaussian eliminations, we obtain that $v(f) \geq v(u_0^{\alpha}f_\alpha)$ for all $\alpha$ and hence
\[
v(f) = \min_\alpha\{v(u_0^{\alpha}f_\alpha)\}.
\]
Furthermore, since $v(u_1) = 0$ and $v(u_2) > 0$, every monomial in $f_\alpha$ has distinct $v$-valuation and $v(f_\alpha)$ is determined by the monomial in $f_\alpha$ with the lowest power in $u_2$. In particular,
\[
v(f) = \min\{v(u_0^{\alpha}u_1^{-\beta}u_2^{\gamma}): u_0^{\alpha}u_1^{-\beta}u_2^{\gamma}\text{ is a monomial in }f\},
\]
which implies that $v$ is a torus-invariant valuation.
\end{proof}

\begin{lemma}\label{lemma:MMP_isomorphism_away_from_L0}
In diagram (\ref{equation:construction_of_special_degeneration}), the birational maps $\mathcal{Y}\dashrightarrow \mathcal{X}'$ and $\mathcal{X}'\dashrightarrow \mathcal{X}^m$ are isomorphisms in a neighborhood of the closure of $\{x\}\times (\AA^1\setminus \{0\})$.
\end{lemma}
\begin{proof}
It suffices to show that on the central fibers, the rational maps $S_1\dashrightarrow S_2$ and $S_2\dashrightarrow S_3$ are isomorphisms in a neighborhood of $x_i\in S_i$. By Lemma~\ref{lemma:factorize_MMP_into_flips_and_divisorial_contraction}, $S_1\dashrightarrow S_2$ is an isomorphism away from $L_0$, so it is an isomorphism in a neighborhood of $x_1 = [1,0,0]\not\in L_0$.

We note that $(\mathcal{X}', \mathcal{C}_{\mathcal{X}'} + S_2)$ is crepant birational to the pair $(X\times \AA^1, C\times \AA^1 + X\times \{0\})$, so it is log canonical and log Calabi-Yau over $\AA^1$. By adjunction, $(S_2, C_{S_2})$ is a log canonical and log Calabi-Yau pair. Thus, $(S_2, C_{S_2})$ is crepant birational to $(S_1, L_0 + C_{S_1} = L_0 + L_1 + L_2)$ and hence $C_{S_2}\geq L_{1,S_2} + L_{2,S_2}$. In particular, $(S_2, C_{S_2})$ is an SNC pair in a neighborhood of $x_2$. Since the MMP on $\mathcal{X}'$ as in Construction~\ref{construction: special_degeneration} is a MMP over $\AA^1$ with divisor $-K_{\mathcal{X}'}\sim_{\QQ, \AA^1} \mathcal{C}_{\mathcal{X}'}$, each step of this MMP only contracts components of $C_{S_2}$. Thus, it suffices to prove that neither $L_{1,S_3}$ nor $L_{2,S_3}$ is contracted.

Since $E$ is a special divisor, by Theorem~\ref{thm:special_divisor_criterion}, we may choose a $\QQ$-divisor $\Delta'\sim_\QQ -K_X$ such that $E$ is a log canonical place of $(X,\Delta')$ and $\Delta'\geq \epsilon G$ for a general ample divisor $G\sim_\QQ -K_X$ on $X$. Since $\{x\}\times \AA^1$ is a log canonical center of $(K_{X\times \AA^1} + \Delta'\times \AA^1 + X\times \{0\})$,  $x_{1}$ is a log canonical center of $(K_{\mathcal{Y}} + \Delta'_{\mathcal{Y}} + S_1)$. Write $\Delta' = (1-\epsilon) C' + \epsilon G$. Since $G$ is general, $\mathcal{G}_{\mathcal{Y}}$ does not contain $x_1$, where $\mathcal{G}_{\mathcal{Y}}$ is the strict transform of $G\times \AA^1$ on $\mathcal{Y}$. Then, $x_1$ is a log canonical center of $(K_{\mathcal{Y}} + (1-\epsilon)\mathcal{C}'_{\mathcal{Y}} + S_1)$, where $\mathcal{C}'_{\mathcal{Y}}$ is the strict transform of $C'\times \AA^1$. Since $\mathcal{Y}\dashrightarrow \mathcal{X}'$ is an isomorphism in a neighborhood of $x_{1}$, $x_{2}$ is a log canonical center of $(K_{\mathcal{X}'} + (1-\epsilon)\mathcal{C}'_{\mathcal{X}'} + S_2)$. As $(K_{\mathcal{X}'} + \Delta_{\mathcal{X}'}' + S_2)$ is log canonical, $\mathcal{G}_{\mathcal{X}'}$ does not contain $x_{2}$. In particular, $L_{1,S_2}$, $L_{2,S_2}$ are not contained in the support of $\mathcal{G}_{\mathcal{X}'}$. This implies that
\[
-K_{\mathcal{X}'}\cdot L_{i,S_2} = \mathcal{G}_{\mathcal{X}'}\cdot L_{i,S_2} \geq 0.
\]
for $i= 1,2$. Hence, neither $L_{1,S_2}$ nor $L_{2,S_2}$ is contracted in the first step of the $-K_{\mathcal{X}'}$-MMP on $\mathcal{X}'$ as in Construction \ref{construction: special_degeneration}.(3).

By an inductive argument, at each step of the $-K_{\mathcal{X}'}$-MMP, the strict transform of $\mathcal{G}_{\mathcal{X}'}$ does not contain the image of $x_{2}$ and hence does not contain the strict transform of $L_{1,S_2}$ (or $L_{2,S_2}$). Therefore, this MMP is an isomorphism in a neighborhood of $x_2$ and so is $S_2\dashrightarrow S_3$.
\end{proof}

\begin{lemma}\label{lemma:construction_of_tilde_Y}
Let $\phi:\Tilde{X}\to X$ be the $(p,q)$-weighted blow-up of $X$ at $x$ along the two branches of $C$. Then $E$ is the reduced exceptional divisor of $\phi$. Let $E_{\phi}$ be the scheme-theoretic exceptional divisor of $\phi$.
Let $\Tilde{\mathcal{Y}} \to \Tilde{X}\times \AA^1$ be the blow-up of the closed subscheme $E_{\phi}\times \{0\}\hookrightarrow \Tilde{X}\times \AA^1$. Then there exists a morphism $\Tilde{\mathcal{Y}}\to\mathcal{Y}$ which fits into a commutative diagram
\[
\begin{tikzcd}
        \Tilde{\mathcal{Y}} \arrow{r} \arrow{d} & \Tilde{X}\times \AA^1\arrow{d} \\
        \mathcal{Y} \arrow{r} & X\times \AA^1
    \end{tikzcd}
\]
such that the exceptional divisor of $\Tilde{\mathcal{Y}}\to\mathcal{Y}$ is an irreducible divisor $\mathcal{E}$ which is the strict transform of $E\times \AA^1$. Furthermore, the central fibers of this commutative diagram is
\[
\begin{tikzcd}
    \Tilde{X}_1 \cup \Tilde{S}_1  \arrow{r}\arrow{d} & \Tilde{X}\arrow{d}\\
    X_1 \cup S_1 \arrow{r} & X
\end{tikzcd}
\]
where $\Tilde{X}_1 \cong\Tilde{X}\cong X_1$ and $\Tilde{S}_1\to S_1$ is the $(p,q)$-weighted blow-up at $x_1 = [1,0,0]$ along $L_1$ and $L_2$.
\end{lemma}
\begin{proof}
By Remark~\ref{remark: weighted_blow_up}, $\Tilde{X}\to X$ is the blow-up of $X$ along a $\mathfrak{m}_x$-primary ideal $\mathcal{I}\subseteq \OO_X$, such that the preimage of $\mathcal{I}$ on $\Tilde{X}$ is $\OO_{\Tilde{X}}(-E_{\phi})$. Then, $\mathcal{Y}\to X\times \AA^1$ is the blow-up of $X\times \AA^1$ along the ideal sheaf $(\mathcal{I},s)\subseteq \OO_{X\times \AA^1}\cong \OO_X[s]$. The preimage of $(\mathcal{I},s)$ on $\Tilde{X}\times \AA^1$ is the ideal $(\OO_{\Tilde{X}}(-E_{\phi}), s) \subseteq \OO_{\Tilde{X}}[s]$, which is the ideal of the closed subscheme $E_{\phi}\times \{0\}$. Therefore, by the universal property of blow-ups, there exists a birational morphism $\Tilde{\mathcal{Y}}\to \mathcal{Y}$ which fits into a commutative diagram
\[
\begin{tikzcd}
        \Tilde{\mathcal{Y}} \arrow{r} \arrow{d} & \Tilde{X}\times \AA^1\arrow{d} \\
        \mathcal{Y} \arrow{r} & X\times \AA^1
    \end{tikzcd}
\]
The exceptional divisor of $\Tilde{\mathcal{Y}}\to\mathcal{Y}$ is the strict transform of the exceptional divisor of $\Tilde{X}\times \AA^1\to X\times \AA^1$, i.e., the strict transform of $E\times \AA^1$.  

We now consider the central fibers. Since $E_{\phi}\times \{0\}$ is a divisor on the central fiber of $\Tilde{X}\times \AA^1$, the central fiber of $\Tilde{\mathcal{Y}}$ has two components $\Tilde{X}_1$ and $\Tilde{S}_1$, where $\Tilde{X}_1\cong \Tilde{X}$. Note that $X_1$ is also a $(p,q)$-blow up of $X$ at $x$ along the two branches of $C$, so $X_1\cong \Tilde{X} \cong \Tilde{X}_1$. Since $\Tilde{\mathcal{Y}}\to \mathcal{Y}$ is a birational map whose exceptional divisor $\mathcal{E}$ is horizontal over $\AA^1$, $\Tilde{S}_1\to S_1$ is also a projective birational map with exceptional divisor $E_0$ (which is the central fiber of $\mathcal{E}$). We note that $\mathcal{E}$ is a log canonical place of the pair $(\mathcal{Y}, \mathcal{C}_{\mathcal{Y}}+S_1 + X_1)$, so $E_0$ is a log canonical place of $(S_1, L_0 + L_1 + L_2)$ by adjunction. Furthermore, the center of $E_0$ on $S_1$ is $x_1$, so $E_0$ must be the exceptional divisor of a weighted blow-up of $S_1$ at $x$ along $L_1$ and $L_2$. The weights must be $p$ and $q$ because $\Tilde{S}_1\to S_1$ is a $(p,q)$ weighted blow-up on the generic fiber.
\end{proof}

\begin{lemma}\label{lemma:compatibility_of_MMP_in_step_2}
There exists a $\GG_m$-equivariant MMP over $\AA^1$ on $\Tilde{\mathcal{Y}}$ with the divisor
\[
K_{\Tilde{\mathcal{Y}}} +{\mathcal{C}}_{\Tilde{\mathcal{Y}}} + \Tilde{X}_1 + (1-\epsilon)\tilde{S}_1 \sim_{\QQ, \AA^1} - \epsilon \Tilde{S}_1,
\]
for a sufficiently small $\epsilon > 0$. Furthermore, each step of this MMP on $\Tilde{\mathcal{Y}}$ is compatible with the $\GG_m$-equivariant MMP over $\AA^1$ on $\mathcal{Y}$ as in Construction~\ref{construction: special_degeneration}.(2) in the following sense: 

Suppose the MMP in Construction~\ref{construction: special_degeneration}.(2) has the form
\[
\mathcal{Y} = \mathcal{Y}_0 \dashrightarrow \mathcal{Y}_1\dashrightarrow \cdots \dashrightarrow \mathcal{Y}_n = \mathcal{X}'.
\]
Then, we can arrange the MMP on $\Tilde{\mathcal{Y}}$ over $\AA^1$ as
\[
\Tilde{\mathcal{Y}} = \Tilde{\mathcal{Y}}_0 \dashrightarrow \Tilde{\mathcal{Y}}_1 \dashrightarrow \cdots \dashrightarrow \Tilde{\mathcal{Y}}_n = \Tilde{\mathcal{X}'},\]
such that $\Tilde{\mathcal{Y}}_i\to \mathcal{Y}_i$ is a $\GG_m$-equivariant morphism with an irreducible exceptional divisor which is the strict transform of $E\times \AA^1$.
\end{lemma}
\begin{proof}
We run a $\GG_m$-equivariant MMP on $\Tilde{\mathcal{Y}}$ over $\AA^1$ with the divisor
\[
K_{\Tilde{\mathcal{Y}}} +{\mathcal{C}}_{\Tilde{\mathcal{Y}}} + \Tilde{X}_1 + (1-\epsilon)\tilde{S}_1 \sim_{\QQ, \AA^1} - \epsilon \Tilde{S}_1.
\]
Similar to the proof of Lemma~\ref{lemma:MMP_isomorphism_away_from_L0}, this MMP is an isomorphism away from the strict transform of $L_0$ on $\Tilde{S}_1$. Furthermore, $\Tilde{\mathcal{Y}}\to \mathcal{Y}$ is an isomorphism in an open neighborhood of $L_0$ and $\Tilde{S}_1$ is the pullback of $S_1$ when restricting to this open neighborhood. Therefore, we can arrange the MMP on $\Tilde{\mathcal{Y}}$ so that it is compatible with the MMP on $\mathcal{Y}$ as in Construction~\ref{construction: special_degeneration}.(2).
\end{proof}

\begin{lemma}\label{lemma:compatibility_of_MMP_in_step_3}
Let $\Tilde{\mathcal{X}}'$ be the outcome of the MMP on $\mathcal{\mathcal{Y}}$ in Lemma~\ref{lemma:compatibility_of_MMP_in_step_2}. Then there exists a $\GG_m$-equivariant MMP on $\Tilde{\mathcal{X}'}$ over $\AA^1$ with the divisor \[K_{\Tilde{\mathcal{X}'}} + {\mathcal{C}}_{\Tilde{\mathcal{X}'}} + \mathcal{E}_{\Tilde{\mathcal{X}'}} + \epsilon {\mathcal{G}}_{\Tilde{\mathcal{X}'}}\sim_{\QQ,\AA^1} \epsilon {\mathcal{G}}_{\Tilde{\mathcal{X}'}} = \epsilon \Phi^*\mathcal{G}_{\mathcal{X}'}\sim_{\QQ,\AA^1} \epsilon \Phi^*(-K_{\mathcal{X}'}),\]
such that each step of this MMP is compatible with the $\GG_m$-equivariant MMP on $\mathcal{X}'$ over $\AA^1$ as in Construction~\ref{construction: special_degeneration}.(3).
\end{lemma}
\begin{proof}
Let $G\sim_{\QQ}-K_X$ be a general ample divisor on $X$. Let $\mathcal{G}_{\mathcal{X}'}$ and $\mathcal{G}_{\Tilde{\mathcal{X}}'}$ denote the strict transforms of $G\times \AA^1$ on $\mathcal{X}'$ and $\Tilde{\mathcal{X}}'$. Similarly, $\mathcal{E}_{\mathcal{X}'}$ and $\mathcal{E}_{\Tilde{\mathcal{X}}'}$ denote the strict transforms of $E\times \AA^1$ on $\mathcal{X}'$ and $\Tilde{\mathcal{X}}'$. By the proof of Lemma~\ref{lemma:MMP_isomorphism_away_from_L0}, $\mathcal{G}_{\mathcal{X}'}$ does not contain $x_2$. By Lemma~\ref{lemma:compatibility_of_MMP_in_step_2}, $\Tilde{\mathcal{X}'}\to \mathcal{X}'$ is an isomorphism away from the closure of $\{x\}\times (\AA^1-\{0\})$. Thus, $\mathcal{G}_{\Tilde{\mathcal{X}}'} = \Phi_2^*\mathcal{G}_{{\mathcal{X}}'}$, where $\Phi_2$ denotes the morphism $\Tilde{\mathcal{X}}'\to \mathcal{X}'$.

We run a $\GG_m$-equivariant MMP on $\Tilde{\mathcal{X}'}$ over $\AA^1$ with the divisor \[K_{\Tilde{\mathcal{X}'}} + {\mathcal{C}}_{\Tilde{\mathcal{X}'}} + \mathcal{E}_{\Tilde{\mathcal{X}'}} + \epsilon {\mathcal{G}}_{\Tilde{\mathcal{X}'}}\sim_{\QQ,\AA^1} \epsilon {\mathcal{G}}_{\Tilde{\mathcal{X}'}} = \epsilon \Phi_2^*\mathcal{G}_{\mathcal{X}'}\sim_{\QQ,\AA^1} \epsilon \Phi_2^*(-K_{\mathcal{X}'}).\]
Since $\Phi_2$ is an isomorphism in a neighborhood of the support of $\mathcal{G}_{\mathcal{X}'}$, we can arrange this MMP on $\Tilde{\mathcal{X}}'$ to be compatible with the MMP on $\mathcal{X}'$ as in Construction~\ref{construction: special_degeneration}.(3).
\end{proof}

\begin{proof}[Proof of Theorem~\ref{thm:qdlt_Fano_type_model_on_special_degeneration}(a)-(e)]
Let $\Tilde{\mathcal{Y}}$ be the variety constructed in Lemma~\ref{lemma:construction_of_tilde_Y}, $\Tilde{\mathcal{X}}'$ be the outcome of the MMP on $\Tilde{\mathcal{Y}}$ in Lemma~\ref{lemma:compatibility_of_MMP_in_step_2}, and 
$\Tilde{\mathcal{X}}^m$ be the outcome of the MMP on $\Tilde{\mathcal{X}}'$ in Lemma~\ref{lemma:compatibility_of_MMP_in_step_3}. This completes the construction of diagram (\ref{equation:construction_of_special_degeneration}).\\
(a) This follows from definition.\\
(b) This follows from Lemma~\ref{lemma:compatibility_of_MMP_in_step_2} and Lemma~\ref{lemma:compatibility_of_MMP_in_step_3}.\\
(c) This follows from Lemma~\ref{lemma:construction_of_tilde_Y} for $\Tilde{X}\to X$ and $\Tilde{S}_1\to S_1$, and follows from Lemma~\ref{lemma:MMP_isomorphism_away_from_L0} for $\Tilde{S}_2\to S_2$ and $\Tilde{S}_3\to S_3$.\\
(d) This is Lemma~\ref{lemma:MMP_isomorphism_away_from_L0}.\\
(e) By Lemma~\ref{lemma:Y_is_1,p,q_weighted_blow_up}, $(S_1, C_{S_1}) = (S_1, L_1 + L_2)$ is an SNC pair in a neighborhood of $x_1$. By Lemma~\ref{lemma:MMP_isomorphism_away_from_L0}, $S_1\dashrightarrow S_3$ is an isomorphism in a neighborhood of $x_1$. Thus, $(S_3, C_{S_3})$ is an SNC pair in a neighborhood of $x_3$.
\end{proof}

\subsection{Qdlt Fano type model of $X'$}
The goal of this subsection is to prove Theorem~\ref{thm:qdlt_Fano_type_model_on_special_degeneration}(f).

Let $E_{\Tilde{S}_1}$ denote the exceptional divisor of $\Tilde{S}_1\to S_1$, and let $E_{\Tilde{S}_i}$ be the strict transform of $E_{\Tilde{S}_1}$ on $\Tilde{S}_i$ for $i=2,3$. Similar to the proof of Lemma~\ref{lemma:factorize_MMP_into_flips_and_divisorial_contraction}, we can factorize the map $\Tilde{S}_1\dashrightarrow \Tilde{S}_2$ as $\Tilde{S}_1\leftarrow \Tilde{S}_{12}\to \Tilde{S}_2$. Denote $g_1: \Tilde{S}_{12}\to \Tilde{S}_1$ and $g_2: \Tilde{S}_{12}\to \Tilde{S}_2$. Then $g_2$ is a divisorial contraction which contracts the strict transform of $L_0.$ By Lemma~\ref{lemma:compatibility_of_MMP_in_step_2}, there is a morphism $\Tilde{S}_{12}\to S_{12}$ which extracts the strict transform $E_{\Tilde{S}_{12}}$ of $E_{\Tilde{S}_1}$. 

\begin{lemma}\label{lemma:big_and_nef_divisor_on_tilde_S2}
There exists an effective $\QQ$-divisor $D_{\Tilde{S}_2}$ on $\Tilde{S}_2$ satisfying the following properties:
\begin{itemize}
    \item $D_{\Tilde{S}_2}$ is big and nef.
    \item $D_{\Tilde{S}_2}$ is fully supported on $\supp(C_{\Tilde{S}_2} - L_{2,\Tilde{S}_2} - E_{\Tilde{S}_2})$.
    \item There exists an effective divisor $B_{\Tilde{S}_2}$ on $\Tilde{S}_2$ such that $L_{2,\Tilde{S}_2}\cap E_{\Tilde{S}_2}$ is not contained in $\text{supp}(B_{\Tilde{S}_2})$ and
    $D_{\Tilde{S}_2} - \epsilon B_{\Tilde{S}_2}$ is ample for any sufficiently small $\epsilon > 0$.
\end{itemize}
\end{lemma}

\begin{proof}
\noindent\textit{Step 1.} We show that $g_{2*}g_1^*L_{1,\Tilde{S}_1}$ is big and nef.

By Lemma~\ref{lemma:construction_of_tilde_Y}, $\Tilde{S}_1$ is a toric variety whose toric fan is spanned by four vectors: $(0,1), (1,0), (p,q), (-p,-q)$. This implies that $L_{1,\Tilde{S}_1}^2 = 0$, so $g_{2*}g_1^*L_{1,\Tilde{S}_1}$ is nef. Since $g_2$ contracts $L_{0, \Tilde{S}_{12}}$ and 
\[
g_1^*L_{1,\Tilde{S}_1}\cdot L_{0,\Tilde{S}_{12}} = L_{1,\Tilde{S}_1}\cdot L_{0,\Tilde{S}_1} >0,
\]
we can write
\[
g_2^*g_{2*}g_1^*L_{1,\Tilde{S}_1} = g_1^*L_{1,\Tilde{S}_1} + cL_{0, \Tilde{S}_{12}}
\]
for some $c>0$. As a result,
\[
(g_{2*}g_1^*L_{1,\Tilde{S}_1})^2 = (g_1^*L_{1,\Tilde{S}_1} + cL_{0, \Tilde{S}_{12}})^2 = cg_1^*L_{1,\Tilde{S}_1}\cdot L_{0,\Tilde{S}_{12}} > 0.
\]
This shows that $g_{2*}g_1^*L_{1,\Tilde{S}_1}$ is big, as desired.\\

\noindent\textit{Step 2.} We show that there exists a nef divisor $D_{\Tilde{S}_{12}}$ on $\Tilde{S}_{12}$ whose support is $\supp(C_{\Tilde{S}_{12}} - L_{2,\Tilde{S}_{12}} - E_{\Tilde{S}_{12}})$.

Since $(X,C)$ is a normal crossing pair in a neighborhood of $x$, $C_{S_1}, C_{S_{12}}, C_{S_2}$, and $C_{S_3}$ are integral divisors on $S_1, S_{12}, S_2$, and $S_3$. Let $F$ be an exceptional divisor of the $\GG_m$-equivariant morphism $f_1: S_{12}\to S_1$. Then the center of $F$ on $S_1$ is contained in $L_0$. If the center of $F$ on $S_1$ is $[0,1,0]$ or $[0,0,1]$, then $F$ is torus-invariant by Lemma~\ref{lemma:G_m_invariant_divisors_on_P(a,b,c)}, which implies that $F$ is a log canonical place of $(S_1, C_{S_1} + L_0)$. Since
\[
f_1^*(K_{S_1} + C_{S_1} + L_0) = K_{S_{12}} + C_{S_{12}} + L_{0, S_{12}},
\]
$F$ is contained in $\supp(C_{S_{12}})$. If the center of $F$ on $S_1$ is not $[0,1,0]$ or $[0,0,1]$, then $F$ is not a log canonical place of $(S_{12}, C_{S_{12}} + L_{0, S_{12}})$ and hence $F\not\subseteq \supp(C_{S_{12}})$ (because $C_{S_{12}}$ is an integral divisor). Therefore, $C_{S_{12}}$ is the sum of $L_{1,S_{12}}, L_{2,S_{12}}$, and all torus-invariant exceptional divisors of $f_1$. Write
\[
f_1^*L_{0} = L_{0, S_{12}} + M_{S_{12}} + N_{S_{12}},
\]
where the center of every component of ${M}_{S_{12}}$ on $S_1$ is contained in $\{[0,1,0],[0,0,1]\}$, and the center of every component of $N_{S_{12}}$ on $S_1$ is contained in $L_0\setminus \{[0,1,0],[0,0,1]\}$.  Then $\supp({M}_{S_{12}})\subseteq \supp(C_{S_{12}})$ and none of the components of $N_{S_{12}}$ is contained in $\supp(C_{S_{12}})$. Furthermore, we can write
\[
g_1^*L_{0,\Tilde{S}_1} = L_{0,\Tilde{S}_{12}} + M_{\Tilde{S}_{12}} + N_{\Tilde{S}_{12}},
\]
where $M_{\Tilde{S}_{12}}$ and $N_{\Tilde{S}_{12}}$ are the strict transforms of $M_{S_{12}}$ and $N_{S_{12}}$.

Define a divisor $D_{\Tilde{S}_{12}}$ on $\Tilde{S}_{12}$ by
\[
D_{\Tilde{S}_{12}} = g_1^*L_{1,\Tilde{S}_1} + \epsilon (L_{0,\Tilde{S}_{12}}+M_{\Tilde{S}_{12}})
\]
for a sufficiently small $\epsilon > 0$. Then $\supp(D_{\Tilde{S}_{12}}) = \supp(C_{\Tilde{S}_{12}} - L_{2,\Tilde{S}_{12}} - E_{\Tilde{S}_{12}})$. Let $F_{\Tilde{S}_{12}}$ be any irreducible curve on $\Tilde{S}_{12}$. We have the following cases:
\begin{itemize}
\item $F_{\Tilde{S}_{12}}\subseteq \supp(M_{\Tilde{S}_{12}})$. Then $F_{\Tilde{S}_{12}}$ is disjoint from $\supp(N_{\Tilde{S}_{12}})$. Thus,
\[
D_{\Tilde{S}_{12}} \cdot F_{\Tilde{S}_{12}} = g_1^*(L_{1,\Tilde{S}_1} + \epsilon L_{0,\Tilde{S}_1}) = 0.
\]
\item $F_{\Tilde{S}_{12}} = L_{0,\Tilde{S}_{12}}$. Then
\[
D_{\Tilde{S}_{12}} \cdot F_{\Tilde{S}_{12}}  = L_{1,\Tilde{S}_1}\cdot L_{0,\Tilde{S}_1} + \epsilon (L_{0,\Tilde{S}_{12}}+M_{\Tilde{S}_{12}}) \cdot L_{0,\Tilde{S}_2} > 0
\]
since $\epsilon$ is sufficiently small.
\item $F\not\subseteq\supp(L_{0,\Tilde{S}_{12}} + M_{\Tilde{S}_{12}})$. Then $D_{\Tilde{S}_{12}}\cdot F_{\Tilde{S}_{12}} \geq 0$ by the nefness of $L_{1,\Tilde{S}_1}$.
\end{itemize}
As a result, $D_{\Tilde{S}_{12}}$ is nef. \\

\noindent\textit{Step 3.} In this step, we finish the proof of Lemma~\ref{lemma:big_and_nef_divisor_on_tilde_S2}.

Let $D_{\Tilde{S}_2} = g_{2*}D_{\Tilde{S}_{12}}$. By step 2, $D_{\Tilde{S}_2}$ is nef and $\supp(D_{\Tilde{S}_2}) = \supp(C_{\Tilde{S}_2} - L_{2,\Tilde{S}_2} - E_{\Tilde{S}_2})$. By step 1, $D_{\Tilde{S}_2} = g_{2*}g_1^*L_{1,\Tilde{S}_1} + \epsilon M_{\Tilde{S}_2}$ is big. 
It suffices to prove the following statement: Suppose $F_{\Tilde{S}_2}$ is an irreducible curve on $\Tilde{S}_2$ such that $F_{\Tilde{S}_2}\cdot D_{\Tilde{S}_2} = 0$, then $F_{\Tilde{S}_2}$ does not contain the point $y_2:= L_{2,\Tilde{S}_2}\cap E_{\Tilde{S}_2}$.

Assume that $y_2\in F_{\Tilde{S}_2}$. Since $g_2$ is an isomorphism in a neighborhood of $y_2$ (by Lemma~\ref{lemma:compatibility_of_MMP_in_step_2}), the strict transform $F_{\Tilde{S}_{12}}$ of $F_{\Tilde{S}_2}$ on $\Tilde{S}_{12}$ is an irreducible curve which contains $y_{12}$, the preimage of $y_2$. Then
\[
F_{\Tilde{S}_{12}}\cdot D_{\Tilde{S}_{12}} \leq g_2^*F_{\Tilde{S}_{2}} \cdot D_{\Tilde{S}_{12}} = F_{\Tilde{S}_2} \cdot D_{\Tilde{S}_{2}} = 0.
\]
Thus, $F_{\Tilde{S}_{12}}\cdot D_{\Tilde{S}_{12}} = 0$. Because $y_{12}\in F_{\Tilde{S}_{12}}$, $F_{\Tilde{S}_{12}}$ is not contained in the support of $L_{0, \Tilde{S}_{12}} + M_{\Tilde{S}_{12}}$. Then
\[
0 = F_{\Tilde{S}_{12}}\cdot D_{\Tilde{S}_{12}} = F_{\Tilde{S}_{12}}\cdot (g_1^*L_{1,\Tilde{S}_1} + \epsilon(L_{0,\Tilde{S}_{12}}+M_{\Tilde{S}_{12}}))
\]
implies that
\[
F_{\Tilde{S}_{12}}\cdot g_1^*L_{1,\Tilde{S}_1} = F_{\Tilde{S}_{12}} \cdot (L_{0, \Tilde{S}_{12}} + M_{\Tilde{S}_{12}})  =0.
\]
Thus, $F_{\Tilde{S}_1} \cdot L_{1,\Tilde{S}_1} = 0$, where $F_{\Tilde{S}_1} = g_{1*}F_{\Tilde{S}_{12}}$. Since the semiample divisor $L_{1,\Tilde{S}_1} $ induces a fibration $\Tilde{S}_1\to \PP^1$, $F_{\Tilde{S}_1}$ must be a fiber. Since $y_1:=g_1(y_{12})\in F_{\Tilde{S}_1}$ and $y_1 = L_{2,\Tilde{S}_1}\cap E_{\Tilde{S}_1}$, we must have $F_{\Tilde{S}_1} = L_{2,\Tilde{S}_1}$. Hence, $F_{\Tilde{S}_2} = L_{2,\Tilde{S}_2}$.

On the other hand, since $C_{S_{12}}$ is the sum of $L_{1,S_{12}}, L_{2,S_{12}}$, and all torus-invariant exceptional divisors of $f_1:S_{12}\to S_1$,  $C_{S_{12}} + L_{0,S_{12}}$ consists of a circle of rational curves on $S_{12}$. This implies that $C_{\Tilde{S}_2}$ also consists of a circle of rational curves on $\Tilde{S}_2$. Therefore,
\[
(C_{\Tilde{S}_2} - L_{2,\Tilde{S}_2}-E_{\Tilde{S}_2}) \cdot L_{2,\Tilde{S}_2} > 0,
\]
which implies that
\[
D_{\Tilde{S}_2}\cdot F_{\Tilde{S}_2} > 0.
\]
This is a contradiction. Thus, $y_2\not\in F_{\Tilde{S}_2}$ and the proof is complete.
\end{proof}

\begin{lemma}\label{lemma:big_and_nef_divisor_on_tilde_S3}
There exists an effective $\QQ$-divisor $D_{\Tilde{S}_3}$ on $\Tilde{S}_3$ satisfying the following properties:
\begin{itemize}
    \item $D_{\Tilde{S}_3}$ is big and nef.
    \item $D_{\Tilde{S}_3}$ is fully supported on $\supp(C_{\Tilde{S}_3} - L_{2,\Tilde{S}_3} - E_{2,\Tilde{S}_3})$.
    \item There exists an effective divisor $B_{\Tilde{S}_3}$ on $\Tilde{S}_3$ such that $L_{2,\Tilde{S}_3}\cap E_{\Tilde{S}_3}$ is not contained in $\text{supp}(B_{\Tilde{S}_3})$ and
    $D_{\Tilde{S}_3} - \epsilon B_{\Tilde{S}_3}$ is ample for any sufficiently small $\epsilon > 0$.
\end{itemize}
\end{lemma}
\begin{proof}
Let $\Tilde{S}_1\leftarrow \Tilde{S}_{23}\to \Tilde{S}_2$ resolve the indeterminacy locus of the birational map $\Tilde{S}_2\dashrightarrow \Tilde{S}_3$. Denote $g_3: \Tilde{S}_{23}\to \Tilde{S}_2$ and $g_4:\Tilde{S}_{23}\to \Tilde{S}_3$. 

Let $D_{\Tilde{S}_2}$ be the divisor constructed in Lemma~\ref{lemma:big_and_nef_divisor_on_tilde_S2}. Define $D_{\Tilde{S}_3} = g_{4*}g_3^*D_{\Tilde{S}_2}$. Then $D_{\Tilde{S}_3}$ is big and nef. Furthermore, similar to step 3 of the proof of Lemma~\ref{lemma:big_and_nef_divisor_on_tilde_S2}, we can show that
\begin{itemize}
\item $C_{\Tilde{S}_3}$ consists of a circle of rational curves on $\Tilde{S}_3$,
\item $\supp(D_{\Tilde{S}_3}) = \supp(C_{\Tilde{S}_3} - L_{2,\Tilde{S}_3}-E_{\Tilde{S}_3})$, and
\item there does not exist an irreducible curve $F_{\Tilde{S}_3}$ on $\Tilde{S}_3$ such that $F_{\Tilde{S}_3}\cdot D_{\Tilde{S}_3} = 0$ and the point $L_{2,\Tilde{S}_3}\cap E_{\Tilde{S}_3}$ is contained in $F_{\Tilde{S}_3}$.
\end{itemize}
This completes the proof.
\end{proof}

\begin{proof}[Proof of Theorem~\ref{thm:qdlt_Fano_type_model_on_special_degeneration}(f)]
Let $D_{\Tilde{S}_3}$ and $B_{\Tilde{S}_3}$ be the divisors constructed in Lemma~\ref{lemma:big_and_nef_divisor_on_tilde_S3}. For sufficiently small $\epsilon, \epsilon' > 0$, we can write
\[
K_{\Tilde{S}_3} + (C_{\Tilde{S}_3} - L_{2,\Tilde{S}_3} - E_{\Tilde{S}_3} - \epsilon'D_{\Tilde{S}_3}) + \epsilon\epsilon' B_{\Tilde{S}_3} + L_{2,\Tilde{S}_3} + E_{\Tilde{S}_3} \sim_{\QQ} - \epsilon'(D_{\Tilde{S}_3} - \epsilon B_{\Tilde{S}_3}), 
\]
which is anti-ample. Let $\Gamma_{\Tilde{S}_3} = (C_{\Tilde{S}_3} - L_{2,\Tilde{S}_3} - E_{\Tilde{S}_3} - \epsilon'D_{\Tilde{S}_3}) + \epsilon\epsilon' B_{\Tilde{S}_3}\geq 0$. By Lemma~\ref{lemma:big_and_nef_divisor_on_tilde_S3}, $\supp(D_{\Tilde{S}_3}) = \supp(C_{\Tilde{S}_3} - L_{2,\Tilde{S}_3} - E_{\Tilde{S}_3})$ and $\supp(B_{\Tilde{S}_3})$ does not contain $L_{2,\Tilde{S}_3}\cap E_{\Tilde{S}_3}$. By Theorem~\ref{thm:qdlt_Fano_type_model_on_special_degeneration}(e), $(S_3, C_{S_3})$ is a log canonical pair which is normal crossing in a neighborhood of $x_3$. Therefore, the pair $(\Tilde{S}_3, \Gamma_{\Tilde{S}_3} + L_{2,\Tilde{S}_3} + E_{\Tilde{S}_3})$ is qdlt, log Fano, and $\lfloor \Gamma_{\Tilde{S}_3} + L_{2,\Tilde{S}_3} + E_{\Tilde{S}_3}\rfloor =  L_{2,\Tilde{S}_3} + E_{\Tilde{S}_3}$. This implies that $(\Tilde{S}_3, L_{2,\Tilde{S}_3} + E_{\Tilde{S}_3})$ is a qdlt Fano type model of $X'$. 

Similarly, $(\Tilde{S}_3, L_{1,\Tilde{S}_3} + E_{\Tilde{S}_3})$ is also a qdlt Fano type model of $X'$, and the proof is complete.
\end{proof}

\subsection{Local Structure of the Dual Complex of $(X',C')$}
Let $r$ be a positive integer such that $rK_X$ and $rK_{X'}$ are Cartier. Let $R_m = H^0(-mK_X)$ and $R_m' = H^0(-mK_{X'})$ for $m\in r\NN$.  By Proposition~\ref{prop:rees_construction}, there is an isomorphism of graded $\CC$-algebras
\[
\bigoplus_{m\in r \mathbb{N}}R_m' \cong \bigoplus_{m\in r \mathbb{N}}  \gr_E^\bullet R_m.
\]
In fact, this isomorphism is $\GG_m$-equivariant, where the $\GG_m$ action on $H^0(X', -mK_{X'})$ is induced from the $\GG_m$ action on $X'$, and the $\GG_m$ action on $\gr_E^i R_m$ has weight $-i$ for each $i\geq 0$. 

For each $t\in (0,\infty)$, we denote $v_t\in \text{QM}(X,C)$ the quasi-monomial valuation centered at $x\in X$ with weight $(1,t)$ along the two branches of $C$.
By Theorem~\ref{thm:qdlt_Fano_type_model_on_special_degeneration}(e), in a neighborhood of $x_3$, $(S_3, C_{S_3}) = (S_3, L_{1,S_3} + L_{2,S_3})$ is a normal crossing pair. Define $\overline{v_t}\in \text{QM}(S_3, L_{1,S_3} + L_{2,S_3})$ to be the quasi-monomial valuation at $x_3$ with weight $(1,t)$ along $L_{1,S_3}$ and $L_{2,S_3}$. Let $t_0 = \frac{q}{p}$. Then $\ord_E = p\cdot v_{t_0}$ as valuations on $X$. By Theorem~\ref{thm:qdlt_Fano_type_model_on_special_degeneration}(c), $\ord_{E_{\Tilde{S}_3}} = p\cdot \overline{v_{t_0}}$ as valuations on $S_3$.

\begin{definition}
For each $f\in R_m$ with $\ord_E(f) = i$, we denote $\overline{f}\in \gr^i_ER_m$ to be the image of $f$. Using the isomorphism $H^0(X', -mK_{X'}) = R_m' \cong \gr_E^\bullet R_m,$ we may also identify $\overline{f}$ with a section of $|-mK_{X'}|$.
\end{definition}

The main theorem of this subsection is
\begin{theorem}\label{thm:isomorphism_of_graded_rings_for_special_degenerations}
    There exists an open interval $I$ containing $t_0$ such that for all $t\in I$, \[\bigoplus_{m\in r\mathbb{N}}\gr_{v_t}^\bullet R_m\cong \bigoplus_{m\in r\mathbb{N}} \gr_{\overline{v_t}}^\bullet R_m' \]
    as finitely generated graded $\CC$-algebras.
\end{theorem}

\begin{lemma}\label{lemma:v_t_and_bar_v_t_are_the_same_for_a_single_f}
Let $f\in R_m$ for some $m\in r\cdot \mathbb{N}$. For every $t\in (0,\infty)$, we have
\[
v_t(f)\leq \overline{v_t}(\overline{f}).
\]
Furthermore, there exists $t_1 < t_0$ and $t_2 > t_0$ (depending on $f$) such that
\begin{itemize}
\item $v_t(f)=  \overline{v_t}(\overline{f})$ for all $t\in [t_1, t_2]$.
\item There exists $(i_1,j_1)\in \mathbb{N}^2$ such that $v_t(f) = i_1+tj_1$ for all $t\in [t_1, t_0]$.
\item There exists $(i_2,j_2)\in \mathbb{N}^2$ such that $v_t(f) = i_2+tj_2$ for all $t\in [t_0, t_2]$.
\end{itemize}
\end{lemma}
\begin{proof}
Let $x_1,x_2$ be the analytic coordinates on $X$ in a neighborhood of $x$, such that $C$ is cut out by $x_1x_2 = 0$. Write $f = \sum_{i,j\in \mathbb{N}}c_{ij} x_1^i x_2^j$. Denote $\alpha = v_{t_0}(f) = \min\{i+t_0j: c_{ij}\neq 0\}$. Let $D$ be the divisor on $X$ corresponding to $f$. Let $\mathcal{D}_{\mathcal{Y}}$ denote the strict transform of $D\times \AA^1$ on $\mathcal{Y}$ and $D_{S_1} = \mathcal{D}_{\mathcal{Y}}|_{S_1}$. We define $\mathcal{D}_{\mathcal{X}^m}$ and $D_{S_3}$ similarly. Then $\mathcal{D}_{\mathcal{X}^m}\sim_{\QQ, \AA^1} -mK_{\mathcal{X}^m}$ and $D_{S_3}\sim_{\QQ} -mK_{S_3}$.

Recall that $S_1 \cong \PP(1,p,q)$ with homogeneous coordinates $u_0,u_1,u_2$. By a local computation near $x_1$, we can write the equation of $D_{S_1}$ as
\[
D_{S_1} = \left(\sum_{\substack{i,j\in \mathbb{N}\\i+t_0j = \alpha}} c_{ij} u_1^iu_2^j = 0\right).
\]
By Theorem~\ref{thm:qdlt_Fano_type_model_on_special_degeneration}(d), $\mathcal{Y}$ and $\mathcal{X}^m$ are isomorphic in a neighborhood of the closure of $\{x\}\times (\AA^1-\{0\})$. Thus, abusing the notation, we may still write $L_{1,S_3} = (u_1 = 0)$ and $L_{2,S_3} = (u_2 = 0)$ in a neighborhood of $x_{3}$. Thus, the local equation of $D_{S_3}$ has the same form as the local equation of $D_{S_1}$. 

Furthermore, by the diagram (\ref{equation: construction_of_special_degeneration_central_fiber}), $S_3$ is the ample model of $X'$. Thus, $D_{S_3}$ is identified with $\overline{f}$ by the isomorphism $H^0(S_3, -mK_{S_3})\cong H^0(X', -mK_{X'})$. Therefore,
\[
\overline{v_t}(\overline{f}) = \overline{v_t}(D_{S_3}) = \min\{i+tj: c_{ij}\neq 0, i+t_0j = \alpha\} \geq \min \{i+tj: c_{ij}\neq 0\} = v_t(f).
\]

Among all pairs of integers $(i,j)$ with $i+t_0j = \alpha$ and $c_{ij} \neq 0$, let $(i_1,j_1)$ be the pair with the largest value in $j$ and $(i_2,j_2)$ be the pair with the smallest value in $j$. Then, for $t < t_0$ sufficiently close to $t_0$, we have
\[
i_1 + tj_1 = \min \{i+tj: c_{ij}\neq 0\}.
\]
This implies that
\[
v_t(f) = i_1 + tj_1 = \overline{v_t}(\overline{f}).
\]
Similarly, for $t > t_0$ sufficiently close to $t_0$, we have
\[
v_t(f) = i_2 + tj_2 = \overline{v_t}(\overline{f}).
\]
This completes the proof of this lemma.
\end{proof}

\begin{proof}[Proof of Theorem~\ref{thm:isomorphism_of_graded_rings_for_special_degenerations}]
It suffices consider the case $t > t_0$, since the case $t = t_0$ follows from the fact that $v_{t_0}(f) = \overline{v_{t_0}}(\overline{f})$ for all $f$, and the case $t< t_0$ is analogous. \\

\textit{Step 1.} We set up some notations.
By Theorem~\ref{thm:qdlt_fano_type_model_and_higher_rank_degeneration} and Theorem~\ref{thm:qdlt_Fano_type_model_on_special_degeneration}(f), for every $t>t_0$,
\[
\bigoplus_{m\in r\mathbb{N}} \gr_{\overline{v_t}}R_m' = \bigoplus_{m\in r\mathbb{N}} \gr_{\overline{v_t}}^\bullet H^0(X', -mK_{X'})
\]
is a finitely generated algebra. Furthermore, one may choose a set of elements \[\overline{f_1}\in R'_{m_1}\ldots,\overline{f_s}\in R'_{m_s}\] which is independent of $t$, such that their images generate $\bigoplus_{m\in r\mathbb{N}} \gr_{\overline{v_t}}^\bullet R_m'$. We may further assume that $\overline{f_k}\in \gr_E^{\alpha_k}R_{m_k}$ and pick a lifting $f_k\in R_{m_k}$ of $\overline{f_k}$.

By Lemma~\ref{lemma:v_t_and_bar_v_t_are_the_same_for_a_single_f}, there exists $t_2 > t_0$ such that for all $k=1,2,\ldots,s$ and $t\in [t_0, t_2]$,
\[
v_t (f_k) = \overline{v_t}(\overline{f_k}) = i_k +tj_k.
\]
This implies that the local equation of $f_k$ near $x$ can be written as
\[
f_k = x_1^{i_k} x_2^{j_k} + g_k(x_1,x_2)
\]
where $g_k$ is a formal power series in $x_1, x_2$ such that every monomial $x_1^i x_2^j$ with nonzero coefficient in $g_k$ satisfies
\[
i+tj > i_k + tj_k = v_t(f_k).
\]
Similarly, in a local coordinate $u_1,u_2$ on $S_3$, we can write the equation of $\overline{f_k}$ (i.e., the equation of the corresponding divisor on $S_3$) as
\[
\overline{f_k} = u_1^{i_k} u_2^{j_k} + \overline{g_k}(u_1,u_2)
\]
where $ \overline{g_k}$ is a formal power series in $u_1, u_2$ such that every monomial $u_1^i u_2^j$ with nonzero coefficient in $\overline{g_k}$ satisfies
\[
i+tj > i_k + tj_k = \overline{v_t}(\overline{f_k}).
\]
\\ 
\textit{Step 2.} We prove the following statement. Let $M\subseteq \mathbb{N}^3$ be the submonoid generated by $\{(i_k, j_k, m_k): 1\leq k\leq s\}$. Then for any $t\in (t_0,t_2)$, there is an isomorphism
\[
\bigoplus_{m\in r\mathbb{N}} \gr_{\overline{v_t}}^\bullet R_m'\cong \CC[M]
\]
as graded algebras, where the grading in $\CC[M]$ is given by the third coordinate.

Consider the ring homomorphism 
\[
\phi: \CC[M] \to \bigoplus_{m\in r\mathbb{N}} \gr_{\overline{v_t}}^\bullet R_m'
\]
determined by
\[
1\cdot (i_k, j_k, m_k)\mapsto \text{ the image of } \overline{f_k} \text{ in } \gr_{\overline{v_t}}^\bullet R_{m_k}'
\]
We need to show that $\phi$ is well-defined, i.e., if there exist $a_1,\ldots,a_s,b_1,\ldots,b_s\in \mathbb{N}$ such that
\[
\sum_{k=1}^s a_k (i_k,j_k,m_k) = \sum_{k=1}^s b_k (i_k,j_k,m_k)
\]
as elements of $M$, then
\[
\overline{v_t}\left(\prod_{k=1}^s{\overline{f_k}}^{a_k} \right) = \overline{v_t}\left(\prod_{k=1}^s{\overline{f_k}}^{b_k} \right) <\overline{v_t}\left(\prod_{k}{\overline{f_k}}^{a_k} -\prod_{k}{\overline{f_k}}^{b_k} \right),
\]
so that $\prod_{k}{\overline{f_k}}^{a_k} $ and $\prod_{k}{\overline{f_k}}^{b_k} $
have the same image in $\gr_{\overline{v_t}}^\bullet R'$.

Let $i' = \sum_k a_ki_k$ and $j' = \sum_k a_k j_k $. Then by step 1, we can write
\[
\prod_{k=1}^s{\overline{f_k}}^{a_k} = u_1^{i'} u_2^{j'} + \overline{h_a}(u_1,u_2)
\]
such that for any monomial $u_1^iu_2^j$ in $\overline{h_a}$,
\[
i+tj > i' + tj' = \overline{v_t}\left(\prod_{k=1}^s{\overline{f_k}}^{a_k} \right).
\]
Similarly, we can write
\[
\prod_{k=1}^s{\overline{f_k}}^{b_k} = u_1^{i'} u_2^{j'} + \overline{h_b}(u_1,u_2).
\]
Then
\[
\overline{v_t}\left(\prod_{k}{\overline{f_k}}^{a_k} -\prod_{k}{\overline{f_k}}^{b_k}\right) = \overline{v_t}\left(\overline{h_a} - \overline{h_b}\right)>i'+tj'  = \overline{v_t}\left(\prod_{k}{\overline{f_k}}^{a_k}\right),
\]
as desired. This proves that $\phi$ is well-defined.

Next, we show that $\phi$ is injective. Suppose that an element \[d = \sum_{(i,j,m)\in M} c_{i,j,m} (i,j,m)\in \CC[M]\] is mapped to 0 under $\phi$. Since each $1\cdot (i,j,m)$ is mapped to an element in $\gr_{\overline{v_t}}^{i+tj}R_m'$, we only need to consider the case when $m = m_0$ and $i+tj = v$ are both fixed, so that
\[
d = \sum_{(i,j,m_0)\in M: i+tj = v} c_{i,j,m_0} (i,j,m_0).
\]
Then, $\phi(d)$ is the image of the following element in $\gr_{\overline{v_t}}^{v}R_m'$:
\[
\sum_{(i,j,m_0)\in M: i+tj = v} c_{i,j,m_0} (u_1^i u_2^j + h_{i,j}(u_1,u_2)),
\]
where $h_{i,j}$ is a power series in $u_1,u_2$ such that any monomial in $h_{i,j}$ with nonzero coefficient has larger $\overline{v_t}$-valuation than $v$. Since $\phi(d) = 0$, we have
\[
\sum_{(i,j,m_0)\in M: i+tj = v} c_{i,j,m_0} u_1^i u_2^j = 0,
\]
which implies that $c_{i,j,m_0} = 0$ for all $(i,j)$ with $i+tj = v$. Therefore, $d = 0$ and $\phi$ is injective.

Finally, $\phi$ is surjective because the images of $\overline{f_1},\ldots,\overline{f_s}$ generate $\bigoplus_{m\in r\mathbb{N}} \gr_{\overline{v_t}}^\bullet R_m'$. \\

\textit{Step 3.} We prove that for all $t\in (t_0,t_2)$,
\[
\CC[M]\cong \bigoplus_{m\in r\mathbb{N}}\gr_{v_t}^\bullet R_m.
\]

We consider the ring homomorphism
\[
\varphi: \CC[M]\to \bigoplus_{m\in r\mathbb{N}}\gr_{v_t}^\bullet H^0(X, -mK_{X})
\]
determined by
\[
1\cdot (i_k, j_k, m_k)\mapsto \text{ the image of } {f_k} \text{ in } \gr_{{v_t}}^\bullet R_{m_k}.
\]
Similar to the argument in step 2 and using Lemma~\ref{lemma:v_t_and_bar_v_t_are_the_same_for_a_single_f}, one can show that $\varphi$ is well-defined, graded, and injective.

By step 2, $\CC[M]\cong \bigoplus_{m\in r\mathbb{N}}\gr_{\overline{v_t}}^\bullet R_m'$, so we have an equality
\begin{align*}
    \dim_\CC \CC[M]_m &= \dim_\CC \gr_{\overline{v_t}}^\bullet R_m' = \dim_\CC R_m' \\
    &= \dim_\CC R_m = \dim_\CC \gr_{v_t}^\bullet R_m
\end{align*}
for every $m\in r\NN$. Thus, $\varphi$ must be an isomorphism. This finishes the proof of this theorem.
\end{proof}

\begin{remark}
{\em Chenyang Xu informed us that there could be another proof of Theorem~\ref{thm:isomorphism_of_graded_rings_for_special_degenerations} using \cite[Lemma 5.38]{Xu25}}.
\end{remark}

\subsection{Proof of Theorem~\ref{mainthm:local_structure_of_dual_complex_of_special_degenerations}}
\begin{proof}
We follow the notations in Definition~\ref{def:notations_of_construction_of_special_degeneration}. Then $E'$ is the exceptional divisor of $\Tilde{S}_1\to S_1\cong \PP(1,p,q)$, which is the $(p,q)$-weighted blow-up at $x_1 = [1,0,0]$ along $L_1 = (u_1 = 0)$ and $L_2 = (u_2 = 0)$ by Lemma~\ref{lemma:Y_is_1,p,q_weighted_blow_up}. Then Part (a) follows from the fact that $(X',C')$ is crepant birational to the log canonical toric pair $(S_1, L_0 + L_1 + L_2)$.

Parts (b) and (c) follow from Theorem~\ref{thm:isomorphism_of_graded_rings_for_special_degenerations}.
\end{proof}

\section{Examples}

When $X$ is a general smooth del Pezzo surface and $C\sim -K_X$ is a normal crossing divisor, one can explicitly compute the partition of $\mathcal{D}^{KV}(X,C)$ as in Theorem~\ref{mainthm:structure_of_the_space_of_special_valuations} as well as the associated special $\GG_m$-equivariant degenerations of $X$. Then case where $X = \PP^2$ is studied in \cite[Remark 6.6]{LXZ22} and \cite[Appendix A]{ABBDILW23}. Here, we list some results when $X$ is the blow-up of $\PP^2$ at a point and $C$ is an irreducible nodal divisor on $X$. Note that this case does not appear in \cite{HP10} since every $\GG_m$-equivariant degenerations of $X$ has Picard rank at least 2.

\begin{example}

Let $X$ be the blow-up of $\PP^2$ at a point. Let $C\sim -K_X$ be an irreducible nodal curve with node $x$. For $t>0$, let $v_t$ denote the quasi-monomial valuation at $x$ with weight $(1,t)$ along the two branches of $C$. Let $\{g_k: k\in \ZZ\}$ be the sequence such that $(g_0, g_1,\ldots, g_5) = (1,1,1,1,2,4)$ and $g_{k+6} = 6g_{k+3} - g_k$ for all $k$.
\begin{itemize}
\item[(a)] If $t\in \left(\frac{g_{k+3}}{g_k}, \frac{g_{k+4}}{g_{k+1}}\right)$, then $v_t$ induces a special degeneration of $X$ into a fixed toric variety whose toric polytope is a quadrilateral with the following vertices:
\begin{align*}
(0,0), \left(\frac{2g_{k+2}+g_{k+3}}{g_{k+1}}, 0\right), \left(\frac{g_{k+3}}{g_{k+2}}, \frac{2g_{k+1}}{g_{k+2}}\right), \left(0, \frac{g_{k+1}+g_{k+2}}{g_{k+3}}\right), & \text{ if }k\equiv 0 \pmod 3,\\
(0,0),\left(\frac{g_{k+2}+2g_{k+3}}{g_{k+1}}, 0\right), \left(\frac{g_{k+3}}{g_{k+2}}, \frac{g_{k+1}}{g_{k+2}}\right), \left(0, \frac{2g_{k+1}+g_{k+2}}{g_{k+3}}\right), & \text{ if }k\equiv 1 \pmod 3,\\
(0,0),\left(\frac{g_{k+2}+g_{k+3}}{g_{k+1}}, 0\right), \left(\frac{2g_{k+3}}{g_{k+2}}, \frac{g_{k+1}}{g_{k+2}}\right), \left(0, \frac{g_{k+1}+2g_{k+2}}{g_{k+3}}\right), & \text{ if }k\equiv 2 \pmod 3.
\end{align*}
\item [(b)] If $t = \frac{g_{k+3}}{g_k}$, then $v_t$ induces a special degeneration of $X$ into a complete intersection in a weighted projective space:
\begin{align*}
\frac{\PP\left(1, g_k, g_{k+3}, g_{k+1}(g_{k+1}+g_{k+2}), g_{k+2}(g_{k+1}+g_{k+2})\right)}{(x_0x_3 - x_1^{2g_{k+2}}(x_1^{g_{k+3}}+x_2^{g_k}), x_0x_4 - x_2^{2g_{k+1}}(x_1^{g_{k+3}}+x_2^{g_k}))}, &\text{ if }k\equiv 0 \pmod 3,\\
\frac{\PP\left(1, g_k, g_{k+3}, g_{k+1}(2g_{k+1}+g_{k+2}), g_{k+2}(2g_{k+1}+g_{k+2})\right)}{(x_0x_3 - x_1^{g_{k+2}}(x_1^{g_{k+3}}+x_2^{g_k})^2, x_0x_4 - x_2^{g_{k+1}}(x_1^{g_{k+3}}+x_2^{g_k}))}, &\text{ if }k\equiv 1 \pmod 3, \\
\frac{\PP\left(1, g_k, g_{k+3}, g_{k+1}(g_{k+1}+2g_{k+2}), g_{k+2}(g_{k+1}+2g_{k+2})\right)}{(x_0x_3 - x_1^{g_{k+2}}(x_1^{g_{k+3}}+x_2^{g_k}), x_0x_4 - x_2^{g_{k+1}}(x_1^{g_{k+3}}+x_2^{g_k})^2)}, &\text{ if }k\equiv 2 \pmod 3.
\end{align*}
\item [(c)] If $t\not\in \left(3-2\sqrt{2}, 3+2\sqrt{2}\right)$, then $v_t$ is not special over $X$.
\end{itemize}
\end{example}


\bibliographystyle{alpha}
\bibliography{main}

\vspace{0.5cm}
\end{document}